\documentclass[a4paper,11pt,fleqn]{article}
\usepackage[dvipsnames]{xcolor}
\usepackage{pstricks}
\usepackage{enumitem}
\usepackage{amsmath}
\usepackage{amssymb}
\usepackage{pifont}
\usepackage{theorem} 
\usepackage{euscript}
\usepackage{exscale,relsize}
\usepackage{epic,eepic}
\usepackage{multicol}
\usepackage{makeidx}
\usepackage{srcltx}         
\usepackage{url}
\usepackage{charter}
\usepackage{mathrsfs} 
\usepackage{graphicx}
\usepackage{authblk}
\newcommand{\email}[1]{\href{mailto:#1}{\nolinkurl{#1}}}
\usepackage[normalem]{ulem}     
\newlength{\mySubFigSize}
\definecolor{labelkey}{rgb}{0,0.08,0.45}
\definecolor{refkey}{rgb}{0,0.6,0.0}
\definecolor{Brown}{rgb}{0.45,0.0,0.05}
\definecolor{dgreen}{rgb}{0.00,0.49,0.00}
\definecolor{dblue}{rgb}{0,0.08,0.75}
\RequirePackage[dvips,colorlinks,hyperindex]{hyperref}
\hypersetup{linktocpage=true,citecolor=dblue,linkcolor=dgreen}
\renewcommand{\leq}{\ensuremath{\leqslant}}
\renewcommand{\geq}{\ensuremath{\geqslant}}

\newcommand{\Frac}[2]{\displaystyle{\frac{#1}{#2}}} 
\newcommand{\Scal}[2]{\bigg\langle{#1}\;\bigg|\:{#2}\bigg\rangle} 
\newcommand{\scal}[2]{{\langle{{#1}\mid{#2}}\rangle}}

\newcommand{\menge}[2]{\big\{{#1}~\big |~{#2}\big\}}

\newcommand{\Argmin}{\ensuremath{\text{Argmin}}}
\newcommand{\HH}{\ensuremath{{\mathcal H}}}

\newcommand{\Sum}{\ensuremath{\displaystyle\sum}}
\newcommand{\Int}{\ensuremath{\displaystyle\int}}
\newcommand{\emp}{\ensuremath{{\varnothing}}}

\newcommand{\Id}{\ensuremath{\operatorname{Id}}\,}

\newcommand{\RR}{\ensuremath{\mathbb{R}}}
\newcommand{\RP}{\ensuremath{\left[0,+\infty\right[}}

\newcommand{\RPP}{\ensuremath{\left]0,+\infty\right[}}

\newcommand{\conv}{\ensuremath{\text{\rm{conv}\,}}}
\newcommand{\aff}{\ensuremath{\text{\rm{aff}\,}}}

\newcommand{\RX}{\ensuremath{\left]-\infty,+\infty\right]}}

\newcommand{\NN}{\ensuremath{\mathbb N}}

\newcommand{\weakly}{\ensuremath{\:\rightharpoonup\:}}
\newcommand{\exi}{\ensuremath{\exists\,}}

\newcommand{\ran}{\ensuremath{\text{\rm ran}\,}}

\newcommand{\zer}{\ensuremath{\text{\rm zer}\,}}
\newcommand{\pinf}{\ensuremath{{+\infty}}}

\newcommand{\dom}{\ensuremath{\text{\rm dom}\,}}

\newcommand{\prox}{\ensuremath{\text{\rm prox}}}
\newcommand{\Fix}{\ensuremath{\text{\rm Fix}\,}}

\newcommand{\gra}{\ensuremath{\text{\rm gra}\,}}

\newcommand{\zeroun}{\ensuremath{\left]0,1\right[}}   
\newcommand{\rzeroun}{\ensuremath{\left]0,1\right]}}   
\newcommand{\lzeroun}{\ensuremath{\left[0,1\right[}}

\newtheorem{theorem}{Theorem}[section]
\newtheorem{lemma}[theorem]{Lemma}
\newtheorem{corollary}[theorem]{Corollary}
\newtheorem{proposition}[theorem]{Proposition}

\theoremstyle{plain}{\theorembodyfont{\rmfamily}%
\newtheorem{example}[theorem]{Example}}
\theoremstyle{plain}{\theorembodyfont{\rmfamily}%
\newtheorem{remark}[theorem]{Remark}}
\theoremstyle{plain}{\theorembodyfont{\rmfamily}%
\newtheorem{algorithm}[theorem]{Algorithm}}
\theoremstyle{plain}{\theorembodyfont{\rmfamily}%
}
\theoremstyle{plain}{\theorembodyfont{\rmfamily}%
}
\theoremstyle{plain}{\theorembodyfont{\rmfamily}%
\newtheorem{definition}[theorem]{Definition}}
\theoremstyle{plain}{\theorembodyfont{\rmfamily}%
\newtheorem{problem}[theorem]{Problem}}

\numberwithin{equation}{section}
\begin{document}

\title{\sffamily\huge Quasinonexpansive Iterations on the 
Affine Hull of Orbits: From Mann's Mean Value Algorithm to 
Inertial Methods\footnote{Contact 
author: P. L. Combettes, {\ttfamily plc@math.ncsu.edu},
phone:+1 (919) 515 2671.}}

\author{Patrick L. Combettes$^1$ and Lilian E. Glaudin$^2$\\
\small $\!^1$North Carolina State University\\
\small Department of Mathematics\\
\small Raleigh, NC 27695-8205, USA\\
\small \email{plc@math.ncsu.edu}\\
\small\medskip
\small $\!^2$Sorbonne Universit\'es--UPMC Univ. Paris 06\\
\small UMR 7598, Laboratoire Jacques-Louis Lions\\
\small F-75005 Paris, France\\
\small \email{glaudin@ljll.math.upmc.fr}\\
}

\date{
{\it\large Dedicated to the memory of Felipe \'Alvarez 1972-2017}}

\maketitle

\vskip 8mm

\begin{abstract} 
\noindent
Fixed point iterations play a central role in the design
and the analysis of a large number of optimization algorithms.
We study a new iterative scheme in which the update is obtained
by applying a composition of quasinonexpansive operators 
to a point in the affine hull of the orbit generated up to the
current iterate. This investigation unifies several 
algorithmic constructs, including Mann's mean value method, 
inertial methods, and multi-layer memoryless methods. 
It also provides a framework for the development of
new algorithms, such as those we propose for solving monotone 
inclusion and minimization problems.
\end{abstract} 

\bigskip


\newpage

\section{Introduction}

Algorithms arising in various branches of optimization can be
efficiently modeled and analyzed as fixed point iterations in a
real Hilbert space $\HH$; see, e.g.,
\cite{Baus96,Livre1,Borw17,Byrn14,Else01,Opti04,Siop15,%
Davi15,Slav13}. Our paper unifies 
three important algorithmic fixed point frameworks that
coexist in the literature: mean value methods, inertial
methods, and multi-layer memoryless methods. 

Let $T\colon\HH\to\HH$ be
an operator with fixed point set $\Fix T$. In 1953, inspired by 
classical results on the summation of divergent series
\cite{Bore01,Feje03,Toep11}, Mann \cite{Mann53} proposed to extend
the standard successive approximation scheme
\begin{equation}
\label{e:1}
x_0\in\HH\quad\text{and}\quad(\forall n\in\NN)\quad x_{n+1}=Tx_n
\end{equation}
to the mean value algorithm
\begin{equation}
\label{e:2}
x_0\in\HH\quad\text{and}\quad(\forall n\in\NN)\quad
x_{n+1}=T\overline{x}_n,\quad\text{where}\quad
\overline{x}_n\in\conv\big(x_j\big)_{0\leq j\leq n}.
\end{equation}
In other words, the operator 
$T$ is not applied to the most current iterate as in the
memoryless (single step) process \eqref{e:1}, but to a 
point in the convex hull of the orbit
$(x_j)_{0\leq j\leq n}$ generated so far. His motivation was 
that, although the sequence generated by \eqref{e:1} may fail 
to converge to a fixed point of $T$, that generated by
\eqref{e:2} can under suitable conditions. This work was
followed by interesting developments and analyses of such mean
value iterations, e.g., 
\cite{Barr16,Borw91,Brez98,Dots70,Groe72,Kugl03,%
Mann79,Outl69,Rhoa74},
especially in the case when $T$ is nonexpansive (1-Lipschitzian)
or merely quasinonexpansive, that is
(this notion was essentially introduced in \cite{Diaz67})
\begin{equation}
\label{e:qne}
(\forall x\in\HH)(\forall y\in\Fix T)\quad\|Tx-y\|\leq\|x-y\|.
\end{equation}
In \cite{Jmaa02}, the asymptotic behavior of the mean value
process
\begin{equation}
\label{e:3}
x_0\in\HH\;\;\text{and}\;\;(\forall n\in\NN)\quad
x_{n+1}=\overline{x}_n+\lambda_n\big(T_n\overline{x}_n+e_n-
\overline{x}_n\big),
\;\;\text{where}\;\;
\overline{x}_n\in\conv\big(x_j\big)_{0\leq j\leq n},
\end{equation}
was investigated under general conditions on the construction of
the averaging process $(\overline{x}_n)_{n\in\NN}$ and the
assumptions that, for every $n\in\NN$, 
$e_n\in\HH$ models a possible error made in 
the computation of $T_n\overline{x}_n$, 
$\lambda_n\in\left]0,2\right[$, and $T_n\colon\HH\to\HH$ is 
firmly quasinonexpansive, i.e., $2T_n-\Id$ is quasinonexpansive or,
equivalently \cite{Livre1},
\begin{equation}
\label{e:4}
(\forall x\in\HH)(\forall y\in\Fix T_n)\quad
\scal{y-T_nx}{x-T_nx}\leq 0.
\end{equation}

The idea of using the past of the orbit generated by an algorithm
can also be found in the work of Polyak \cite{Poly64,Poly77}, 
who drew inspiration from classical multistep methods in numerical
analysis. His motivation was to improve the speed of
convergence over memoryless methods. 
For instance, the classical gradient method \cite{Poly63} 
for minimizing a smooth convex function $f\colon\HH\to\RR$ 
is an explicit discretization of the continuous-time
process $-\dot{x}(t)=\nabla f(x(t))$. Polyak \cite{Poly64} proposed
to consider instead the process 
$-\ddot{x}(t)-\beta\dot{x}(t)=\nabla f(x(t))$, where
$\beta\in\RPP$, and studied the
algorithm resulting from its explicit discretization.
He observed that, from a mechanical viewpoint, 
the term $\ddot{x}(t)$ can be interpreted as 
an inertial component. More generally, for a proper lower
semicontinuous convex function $f\colon\HH\to\RX$, \'Alvarez 
investigated in \cite{Alva00} an implicit discretization of 
the inertial differential inclusion
$-\ddot{x}(t)-\beta\dot{x}(t)\in\partial f(x(t))$, namely
\begin{equation}
\label{e:5}
(\forall n\in\NN)\quad
x_{n+1}=\prox_{\gamma_n f}\,\overline{x}_n+e_n,
\quad\text{where}\quad
\begin{cases}
\overline{x}_n=(1+\eta_n)x_n-\eta_n x_{n-1}\\
\eta_n\in\lzeroun\\
\gamma_n\in\RPP,
\end{cases}
\end{equation}
and where $\prox_f$ is the proximity operator of $f$
\cite{Livre1,Mor62b}. The inertial proximal point algorithm
\eqref{e:5} has been extended in various directions, e.g.,
\cite{Alva01,Botc15,Cham15}; see also \cite{Atto00} for further
motivation in the context of nonconvex minimization problems.

Working from a different perspective, a structured 
extension of \eqref{e:1} involving the composition of $m$
averaged nonexpansive operators was proposed in \cite{Opti04}. 
This $m$-layer algorithm is governed by the memoryless recursion
\begin{equation}
\label{e:6}
(\forall n\in\NN)\quad
x_{n+1}=x_n+\lambda_n\big(T_{1,n}\cdots T_{m,n}x_n+e_n-x_n\big),
\quad\text{where}\quad\lambda_n\in\rzeroun.
\end{equation}
Recall that a nonexpansive operator $T\colon\HH\to\HH$ is 
averaged with constant 
$\alpha\in\zeroun$ if there exists a nonexpansive operator
$R\colon\HH\to\HH$ such that $T=(1-\alpha)\Id+\alpha R$
\cite{Bail78,Livre1}. The multi-layer iteration process \eqref{e:6}
was shown in \cite{Opti04} to provide a synthetic analysis
of various algorithms, in particular in the area of monotone 
operator splitting methods. It was extended in \cite{Jmaa15} to
an overrelaxed method, i.e., one with parameters
$(\lambda_n)_{n\in\NN}$ possibly larger than $1$.

In the literature, the asymptotic analysis of the above
methods has been carried out independently because of
their apparent lack of common structure. In the present paper, we
exhibit a structure that unifies \eqref{e:1}, \eqref{e:2}, 
\eqref{e:3}, \eqref{e:5}, and \eqref{e:6} in a single algorithm of 
the form 
\begin{multline}
\label{e:8}
x_0\in\HH\quad\text{and}\quad(\forall n\in\NN)\quad
x_{n+1}=\overline{x}_n+\lambda_n\big(T_{1,n}\cdots
T_{m,n}\overline{x}_n+e_n-\overline{x}_n\big),\\
\quad\text{where}\quad
\overline{x}_n\in\aff\big(x_j\big)_{0\leq j\leq n}
\quad\text{and}\quad\lambda_n\in\RPP,
\end{multline}
under the assumption that each operator $T_{i,n}$ is
$\alpha_{i,n}$-averaged quasinonexpansive, i.e., 
\begin{multline}
\label{e:qave}
(\forall x\in\HH)(\forall y\in\Fix T_{i,n})\\
2(1-\alpha_{i,n})\scal{y-T_{i,n}x}{x-T_{i,n}x}
\leq(2\alpha_{i,n}-1)\big(\|x-y\|^2-\|T_{i,n}x-y\|^2\big),
\end{multline}
for some $\alpha_{i,n}\in\rzeroun$, which means that the operator 
$(1-1/\alpha_{i,n})\Id+(1/\alpha_{i,n})T_{i,n}$ 
is quasinonexpansive. In words, at iteration $n$, a point
$\overline{x}_n$ is picked in the affine hull of the orbit
$(x_j)_{0\leq j\leq n}$ generated so far, a composition of
quasinonexpansive operators is applied to it, up to some error
$e_n$, and the update $x_{n+1}$ is obtained via a relaxation with
parameter $\lambda_n$. Note that \eqref{e:8}--\eqref{e:qave} 
not only brings
together mean value iterations, inertial methods, and the
memoryless multi-layer setting of \cite{Opti04,Jmaa15}, but
also provides a flexible framework to design new iterative
methods.

The fixed point problem under consideration will be the
following (note that we allow $1$ as an averaging constant for
added flexibility).

\begin{problem}
\label{prob:1}
\rm
Let $m$ be a strictly positive integer. 
For every $n\in\NN$ and every $i\in\{1,\ldots,m\}$, 
$\alpha_{i,n}\in\rzeroun$ and $T_{i,n}\colon\HH\to\HH$ is 
$\alpha_{i,n}$-averaged nonexpansive if $i<m$, and 
$\alpha_{m,n}$-averaged quasinonexpansive if $i=m$.
In addition, 
\begin{equation}
\label{e:roma2013-07-04y}
S=\bigcap_{n\in\NN}\Fix T_n\neq\emp,\quad\text{where}\quad
(\forall n\in\NN)\quad T_n=T_{1,n}\cdots T_{m,n},
\end{equation}
and one of the following holds:
\begin{enumerate}[label=\rm(\alph*)]
\itemsep0mm 
\item
\label{prob:1a}
For every $n\in\NN$, $T_{m,n}$ is $\alpha_{m,n}$-averaged 
nonexpansive.
\item
\label{prob:1b}
$m>1$ and, for every $n\in\NN$, $\alpha_{m,n}<1$ and 
$\bigcap_{i=1}^m\Fix T_{i,n}\neq\emp$.
\item
\label{prob:1c}
$m=1$.
\end{enumerate}
The problem is to find a point in $S$.
\end{problem}

To solve Problem~\ref{prob:1}, we are going to employ \eqref{e:8},
which we now formulate more formally.

\begin{algorithm}
\label{algo:1}
Consider the setting of Problem~\ref{prob:1}.
For every $n\in\NN$, let $\phi_n$ be an
averaging constant of $T_n$, let
$\lambda_n\in\left]0,1/\phi_n\right]$ 
and, for every $i\in\{1,\ldots,m\}$, let $e_{i,n}\in\HH$. Let 
$(\mu_{n,j})_{n\in\NN, 0\leq j\leq n}$ be a real array which
satisfies the following:
\begin{enumerate}[label=\rm(\alph*)]
\itemsep0mm 
\item
\label{algo:1a}
$\sup_{n\in\NN}\sum_{j=0}^n |\mu_{n,j}|<\pinf$.
\item
\label{algo:1b}
$(\forall n\in\NN)$ $\sum_{j=0}^n\mu_{n,j}=1$.
\item
\label{algo:1c}
$(\forall j\in\NN)$
$\lim\limits_{\substack{n\to\pinf}}\mu_{n,j}=0$.
\item
\label{algo:1e}
There exists a sequence $(\chi_n)_{n\in\NN}$ in $\RPP$ such
that $\inf_{n\in\NN}\chi_n>0$ and every sequence
$(\xi_n)_{n\in\NN}$ in $\RP$ that satisfies
\begin{equation}
\Big(\exi(\varepsilon_n)_{n\in\NN}\in\RP^\NN\Big)\quad
\begin{cases}
\sum_{n\in\NN}\chi_n\varepsilon_n<\pinf\\
(\forall n\in\NN)\quad 
\xi_{n+1}\leq\sum_{j=0}^n\mu_{n,j}\xi_j+\varepsilon_n
\end{cases}
\end{equation} 
converges.
\end{enumerate}
Let $x_0\in\HH$ and set 
\begin{equation}
\label{Kj8Ygf2-21s}
\begin{array}{l}
\text{for}\;n=0,1,\ldots\\
\left\lfloor
\begin{array}{l}
\overline{x}_n=\Sum_{j=0}^n\mu_{n,j}x_j\\
x_{n+1}=\overline{x}_n\!+\!\lambda_n\Big(T_{1,n}\Big(T_{2,n}
\big(\cdots T_{m-1,n}(T_{m,n}\overline{x}_n\!+\!e_{m,n})
\!+\!e_{m-1,n}\cdots\big)\!+\!e_{2,n}\Big)
\!+\!e_{1,n}\!-\!\overline{x}_n\Big).
\end{array}
\right.\\
\end{array}
\end{equation}
\end{algorithm}

\begin{remark}
\label{r:t_n}
Here are some comments about the parameters appearing in
Problem~\ref{prob:1} and Algorithm~\ref{algo:1}.
\begin{enumerate}
\item
\label{r:t_nii}
The composite operator $T_n$ of \eqref{e:roma2013-07-04y} is 
averaged quasinonexpansive with constant 
\begin{equation}
\label{e:phi_n}
\phi_n=
\begin{cases}
\Bigg(1+\Bigg(\Sum_{i=1}^m
\dfrac{\alpha_{i,n}}{1-\alpha_{i,n}}\Bigg)^{-1}\Bigg)^{-1},
&\text{if}\;\:\displaystyle{\max_{1\leq i\leq m}}\alpha_{i,n}<1;\\
1,&\text{otherwise.}
\end{cases}
\end{equation}
The proof is given in \cite[Proposition~2.5]{Jmaa15} for case
\ref{prob:1a} of Problem~\ref{prob:1}. It easily extends to
case \ref{prob:1b}, while case \ref{prob:1c} is trivial.
\item
\label{r:t_ni}
Examples of arrays $(\mu_{n,j})_{n\in\NN, 0\leq j\leq n}$ that
satisfy conditions \ref{algo:1a}--\ref{algo:1e} in
Algorithm~\ref{algo:1} are provided in \cite[Section~2]{Jmaa02}
in the case of mean value iterations, i.e., 
$\inf_{n\in\NN}\min_{0\leq j\leq n}\mu_{n,j}\geq 0$, with
$\chi_n\equiv 1$. 
An important instance with negative
coefficients will be presented in Example~\ref{ex:main}. 
\item
The term $e_{i,n}$ in \eqref{Kj8Ygf2-21s} models a possible
numerical error in the implementation of the operator $T_{i,n}$.
\end{enumerate}
\end{remark}

The material is organized as follows. In Section~\ref{sec:2} we
provide preliminary results. The main results on the convergence of
the orbits of Algorithm~\ref{algo:1} are presented in
Section~\ref{sec:3}. Section~\ref{sec:4} is dedicated to new
algorithms for fixed point computation,
monotone operator splitting, and nonsmooth minimization based on
the proposed framework. 

\noindent
{\bfseries Notation.}
$\HH$ is a real Hilbert space with scalar
product $\scal{\cdot}{\cdot}$ and associated norm $\|\cdot\|$. We
denote by $\Id$ the identity operator on $\HH$; $\weakly$ and
$\to$ denote, respectively, weak and strong convergence in $\HH$.
The positive and negative parts of $\xi\in\RR$ are respectively
$\xi^+=\max\{0,\xi\}$ and $\xi^-=-\min\{0,\xi\}$. Finally,
$\delta_{n,j}$ is the Kronecker delta: it takes on the value $1$
if $n=j$, and $0$ otherwise.

\section{Preliminary results}
\label{sec:2}

In this section we establish some technical facts that will be
used subsequently. We start with a Gr\"onwall-type result.
\begin{lemma}
\label{l:gr}
Let $(\theta_n)_{n\in\NN}$ and $(\varepsilon_n)_{n\in\NN}$
be sequences in $\RP$, and let $(\nu_n)_{n\in\NN}$ be a
sequence in $\RR$ such that
$(\forall n\in\NN)$ $\theta_{n+1}\leq (1+\nu_n)
\theta_n+\varepsilon_n$.
Then
\begin{equation}
\label{c:gronwall}
(\forall n\in\NN)\quad\theta_{n+1} 
\leq\theta_0\exp\Bigg(\Sum_{k=0}^{n}\nu_k\Bigg)+
\sum_{j=0}^{n-1}\varepsilon_j
\exp\Bigg(\Sum_{k=j+1}^{n}\nu_k\Bigg)+\varepsilon_n.
\end{equation}
\end{lemma}
\begin{proof}
We have $(\forall n\in\NN)$ $1+\nu_n\leq\exp(\nu_n)$. Therefore
$\theta_1\leq\theta_0\exp(\nu_0)+\varepsilon_0$ and
\begin{align}
(\forall n\in\NN\smallsetminus\{0\})\quad\theta_{n+1} 
&\leq\theta_{n}\exp(\nu_{n})+\varepsilon_{n}\nonumber\\
&\leq\theta_{n-1}\exp(\nu_{n})
\exp(\nu_{n-1})+
\varepsilon_{n-1}\exp(\nu_{n})+\varepsilon_{n}
\nonumber\\
&\leq\theta_0\prod_{k=0}^{n}\exp(\nu_{k})+
\Sum_{j=0}^{n-1}\varepsilon_j\prod_{k=j+1}^{n}\exp
(\nu_k)+\varepsilon_n\nonumber\\
&=\theta_0\exp\Bigg(\Sum_{k=0}^{n}\nu_k\Bigg)
+\sum_{j=0}^{n-1}\varepsilon_j
\exp\Bigg(\Sum_{k=j+1}^{n}\nu_k\Bigg)+\varepsilon_n,
\end{align}
as claimed.
\end{proof}

\begin{lemma}{\rm\cite[Theorem~43.5]{Knop54}}
\label{l:lem2}
Let $(\xi_n)_{n\in\NN}$ be a sequence in $\RR$, let $\xi\in\RR$,
suppose that $(\mu_{n,j})_{n\in\NN, 0\leq j\leq n}$ is a real
array that satisfies conditions \ref{algo:1a}--\ref{algo:1c} in
Algorithm~\ref{algo:1}. Then $\xi_n\to\xi$ $\Rightarrow$ 
$\sum_{j=0}^n\mu_{n,j}\xi_j\to\xi$.
\end{lemma}

\begin{lemma}
\label{l:1}
Let $(\beta_n)_{n\in\NN}$, $(\gamma_n)_{n\in\NN}$, 
$(\delta_n)_{n\in\NN}$, $(\eta_n)_{n\in\NN}$, and 
$(\lambda_n)_{n\in\NN}$ be sequences in $\left[0,\pinf\right[$, 
let $(\phi_n)_{n\in\NN}$ be a sequence in $\left]0,1\right]$, 
let $(\vartheta,\sigma)\in\RPP^2$, and let $\eta\in\zeroun$.
Set $\beta_{-1}=\beta_0$ and  
\begin{equation}
(\forall n\in\NN)\quad\omega_n=\dfrac{1}{\phi_n}-\lambda_n,
\end{equation}
and suppose that the following hold:
\begin{enumerate}[label=\rm(\alph*)]
\item 
\label{h:l1c}
$(\forall n\in\NN)\quad\eta_{n}\leq\eta_{n+1}\leq\eta$.
\item 
\label{h:l1b}
$(\forall n\in\NN)\quad\gamma_{n}\leq 
\eta(1+\eta)+\eta\vartheta\omega_n$.
\item 
\label{h:l1a}
$(\forall n\in\NN)\quad  
\dfrac{\eta^2(1+\eta)+\eta\sigma}{\vartheta}
<\dfrac{1}{\phi_n}-\eta^2\omega_{n+1}$. 
\item 
\label{h:l1d}
$(\forall n\in\NN)\quad 0<\lambda_n\leq
\dfrac{\vartheta/\phi_n-
\eta\big(\eta(1+\eta)+\eta\vartheta\omega_{n+1}+\sigma\big)}
{\vartheta\big(1+\eta(1+\eta)+\eta\vartheta\omega_{n+1}
+\sigma\big)} $.
\item 
\label{h:l1e}
$(\forall n\in\NN)\quad 
\beta_{n+1}-\beta_n-\eta_n(\beta_n-\beta_{n-1})\leq 
\dfrac{(1/\phi_n-\lambda_n)\big(\eta_n/(\eta_n+\vartheta\lambda_n)
-1\big)}{\lambda_n}\delta_{n+1}+\gamma_n\delta_{n}$.
\end{enumerate}
Then $\sum_{n\in\NN}\delta_n<\pinf$.
\end{lemma}
\begin{proof}
We use arguments similar to those used in \cite{Alva01,Botc15}.
It follows from \ref{h:l1a} that 
$(\forall n\in\NN)$ $0<\vartheta/\phi_n-
\eta^2\omega_{n+1}\vartheta-\eta^2(1+\eta)-\eta\sigma$. This shows
that $(\lambda_n)_{n\in\NN}$ is well defined. Now set 
$(\forall n\in\NN)$ $\rho_n=1/(\eta_n+\vartheta\lambda_n)$ 
and $\kappa_n=\beta_n-\eta_n\beta_{n-1}+\gamma_n\delta_n$. 
We derive from \ref{h:l1c} and \ref{h:l1e} that 
\begin{align}
\label{e:l24a}
(\forall n\in\NN)\quad\kappa_{n+1}-\kappa_n 
&\leq\beta_{n+1}-\eta_n\beta_n-\beta_n+\eta_n\beta_{n-1} 
+\gamma_{n+1}\delta_{n+1}-\gamma_n\delta_n\nonumber\\
&\leq\bigg(\dfrac{(1/\phi_n-\lambda_n)(\eta_n\rho_n-1)}{\lambda_n}
+\gamma_{n+1}\bigg)\delta_{n+1}.
\end{align}
On the other hand, $(\forall n\in\NN)$
$\vartheta(1+(\eta(1+\eta)+\eta\vartheta\omega_{n+1}+\sigma))>0$.
Consequently, \ref{h:l1d} can be written as 
\begin{equation}
\label{e:l24abis}
(\forall n\in\NN)\quad\vartheta\lambda_n+
\vartheta\lambda_n\big(\eta(1+\eta)+\eta\vartheta\omega_{n+1}
+\sigma\big)\leq\dfrac{\vartheta}{\phi_n}-
\eta\big(\eta(1+\eta)+\eta\vartheta\omega_{n+1}+\sigma\big).
\end{equation}
Using \ref{h:l1c} and \ref{h:l1b}, and then  
\eqref{e:l24abis}, we get
\begin{equation}
\label{e:l24b}
(\forall n\in\NN)\quad
(\eta_n+\vartheta\lambda_n)(\gamma_{n+1}+\sigma)+\vartheta\lambda_n 
\leq(\eta+\vartheta\lambda_n)\big(\eta(1+\eta)
+\eta\vartheta\omega_{n+1}+\sigma\big)+
\vartheta\lambda_n\leq\dfrac{\vartheta}{\phi_n}.
\end{equation}
However, 
\begin{align}
(\forall n\in\NN)\quad&
(\eta_n+\vartheta\lambda_n)(\gamma_{n+1}+\sigma)
+{\vartheta\lambda_n}\leq\dfrac{\vartheta}{\phi_n}\nonumber\\
&\qquad\Leftrightarrow
(\eta_{n}+\vartheta\lambda_n)(\gamma_{n+1}+\sigma)-
(1/\phi_n-\lambda_n)
\vartheta\leq 0\nonumber\\
&\qquad\Leftrightarrow
(1/\phi_n-\lambda_n)
\bigg(\dfrac{-\vartheta}{\eta_n+\vartheta\lambda_n}\bigg)
\leq -(\gamma_{n+1}+\sigma)\nonumber\\
&\qquad\Leftrightarrow 
\dfrac{(1/\phi_n-\lambda_n)(\eta_n\rho_n-1)}{\lambda_n}
+\gamma_{n+1}\leq-\sigma.
\label{e:l24c}
\end{align}
It therefore follows from \eqref{e:l24a} and \eqref{e:l24b} that
\begin{equation}
\label{e:l24d}
(\forall n\in\NN)\quad\kappa_{n+1}-\kappa_n\leq-\sigma\delta_{n+1}.
\end{equation}
Thus, $(\kappa_n)_{n\in\NN}$ is decreasing and 
\begin{equation}
\label{e:l24e}
(\forall n\in\NN)\quad\beta_n-\eta\beta_{n-1}
=\kappa_n-\gamma_n\delta_n
\leq\kappa_n\leq\kappa_0,
\end{equation}
from which we infer that
$(\forall n\in\NN)$
$\beta_n\leq\kappa_0+\eta\beta_{n-1}$. In turn,
\begin{equation}
\label{e:l24f}
(\forall n\in\NN\smallsetminus\{0\})\quad\beta_n\leq
\eta^n\beta_0+\kappa_0\Sum_{j=0}^{n-1}\eta^j 
\leq\eta^n\beta_0+\dfrac{\kappa_0}{1-\eta}.
\end{equation}
Altogether, we derive from \eqref{e:l24d}, 
\eqref{e:l24e}, and \eqref{e:l24f} that
\begin{equation}
(\forall n\in\NN)\quad\sigma
\Sum_{j=0}^n\delta_{j+1}\leq\kappa_0-\kappa_{n+1} 
\leq\kappa_0+\eta\beta_n\leq\dfrac{\kappa_0}{1-\eta} 
+\eta^{n+1}\beta_0.
\end{equation}
Hence, 
$\sum_{j\geq 1}\delta_j\leq{\kappa_0}/((1-\eta)\sigma)<\pinf$,
and the proof is complete.
\end{proof}

\begin{lemma}
\label{l:chi}
Let $(\eta_n)_{n\in\NN}$ be a sequence in $\lzeroun$. For every
$n\in\NN$, set
\begin{equation}
\label{24g9h8Hkj2-09C}
(\forall k\in\NN)\quad
\zeta_{k,n}=
\begin{cases}
0,&\text{if}\;\:k\leq n;\\
\Sum_{j=n+1}^{k}(\eta_j-1),&\text{if}\;\:k>n,
\end{cases}
\end{equation}
and $\chi_n=\sum_{k\geq n}\exp(\zeta_{k,n})$.
Then the following hold:
\begin{enumerate}
\item 
\label{l:chi_i}
Let $\tau\in\left[2,\pinf\right[$ and suppose that 
$(\forall n\in\NN)$
$\eta_{n+1}=n/(n+1+\tau)$. 
Then $(\forall n\in\NN)$ $\chi_n\leq (n+7)/2$.
\item 
\label{l:chi_ii}
Suppose that $(\exi\eta\in\lzeroun)(\forall n\in\NN)$ 
$\eta_n\leq\eta$.
Then $(\forall n\in\NN)\;\chi_n\leq e/(1-\eta)$.
\end{enumerate}
\end{lemma}
\begin{proof}
\ref{l:chi_i}:
We have
$(\forall n\in\NN)(\forall k\in\{n+1,n+2,\ldots\})$
$\zeta_{k,n}=-(1+\tau)\sum_{j=n+1}^{k}1/(j+\tau)\leq
-3\sum_{j=n+1}^{k}1/(j+2)$.
Since $\xi\mapsto 1/(\xi+2)$ is decreasing on
$\left[1,\pinf\right[$, it follows that
\begin{equation}
\label{e:tfc1}
(\forall n\in\NN)(\forall k\in\{n+1,n+2,\ldots\})\quad\zeta_{k,n}
\leq-3\Int_{n+1}^{k+1}\dfrac{d\xi}{\xi+2} 
=\ln\dfrac{(n+3)^3}{(k+3)^3}.
\end{equation}
Furthermore, since $\xi\mapsto 1/(\xi+3)^3$ is decreasing on
$\left]-1,\pinf\right[$, \eqref{24g9h8Hkj2-09C} yields
\begin{equation}
\label{e:tfc2}
(\forall n\in\NN)\quad\chi_n 
\leq\Sum_{k\geq n}\dfrac{(n+3)^3}{(k+3)^3}
\leq (n+3)^3\Int_{n-1}^{\pinf}\dfrac{d\xi}{(\xi+3)^3}
=\dfrac{(n+3)^3}{2(n+2)^2}\leq\dfrac{n+7}{2}.
\end{equation}  

\ref{l:chi_ii}: Note that
\begin{equation}
(\forall n\in\NN)(\forall k\in\{n+1,n+2,\ldots\})\quad\zeta_{k,n}
=\Sum_{j=n+1}^{k}(\eta_j-1)\leq\Sum_{j=n+1}^k(\eta-1)=(\eta-1)(k-n).
\end{equation}
Since $\xi\mapsto\exp((\eta-1)\xi)$ is decreasing on
$\left]-1,\pinf\right[$, it follows that
\begin{equation}
(\forall n\in\NN)\quad\chi_n
\leq\Sum_{k\geq n}\exp\big((\eta-1)(k-n)\big)
\leq\Int_{n-1}^\pinf\exp\big((\eta-1)(\xi-n)\big)d\xi
=\dfrac{\exp(1-\eta)}{1-\eta},
\end{equation}
which proves the assertion. 
\end{proof}

The next example provides an instance of an array
$(\mu_{n,j})_{n\in\NN, 0\leq j\leq n}$ satisfying the conditions of
Algorithm~\ref{algo:1} with negative entries. This example
will be central to the study of the convergence of some inertial
methods.

\begin{example}
\label{ex:main}
Let $(\mu_{n,j})_{n\in\NN, 0\leq j\leq n}$ be a real array 
such that $\mu_{0,0}=1$ and 
\begin{equation}
\label{24g9h8Hkj2-02}
(\forall n\in\NN)\quad 1\leq\mu_{n,n}< 2
\quad\text{and}\quad(\forall j\in\{0,\ldots,n\})\quad
\mu_{n,j}=
\begin{cases}
1-\mu_{n,n},&\text{if}\;\: j=n-1;\\
0,&\text{if}\;\: j<n-1.
\end{cases}
\end{equation}
For every $n\in\NN$, set
\begin{equation}
\label{24g9h8Hkj2-09B}
(\forall k\in\NN)\quad
\zeta_{k,n}=
\begin{cases}
0,&\text{if}\;\: k\leq n;\\
\Sum_{j=n+1}^{k}(\mu_{j,j}-2),&\text{if}\;\:k>n,
\end{cases}
\end{equation}
and suppose that $\chi_n=\sum_{k\geq n}\exp(\zeta_{k,n})<\pinf$.
Then $(\mu_{n,j})_{n\in\NN, 0\leq j\leq n}$ satisfies conditions
\ref{algo:1a}--\ref{algo:1e} in Algorithm~\ref{algo:1}.
\end{example}
\begin{proof}
\ref{algo:1a}: $(\forall n\in\NN)$ $\sum_{j=0}^n |\mu_{n,j}|=
\mu_{n,n}+|1-\mu_{n,n}|\leq 3$.

\ref{algo:1b}: $(\forall n\in\NN)$ 
$\sum_{j=0}^n\mu_{n,j}=(1-\mu_{n,n})+\mu_{n,n}=1$.

\ref{algo:1c}: Let $j\in\NN$. Then $(\forall n\in\NN)$ 
$n>j+1\Rightarrow\mu_{n,j}=0$. Hence, 
$\lim\limits_{\substack{n\to\pinf}}\mu_{n,j}=0$.

\ref{algo:1e}: We have $(\forall n\in\NN)$ 
$\chi_n=\sum_{k\geq n}\exp(\zeta_{k,n})\geq\exp(\zeta_{n,n})=1$.
Now suppose that $(\xi_n)_{n\in\NN}$ is a sequence in $\RP$ such 
that there exists a sequence $(\varepsilon_n)_{n\in\NN}$ in $\RP$ 
that satisfies
\begin{equation}
\label{24g9h8Hkj2-01}
\Sum_{n\in\NN}\chi_n\varepsilon_n<\pinf\quad\text{and}\quad
(\forall n\in\NN)\quad
\xi_{n+1}\leq\Sum_{j=0}^n\mu_{n,j}\xi_j+\varepsilon_n.
\end{equation}
Set $\theta_0=0$ and $(\forall n\in\NN)$
$\theta_{n+1}=[\xi_{n+1}-\xi_{n}]^+$ and $\nu_n=\mu_{n,n}-2$. It 
results from \eqref{24g9h8Hkj2-01}, \eqref{24g9h8Hkj2-02}, and the 
inequalities $\xi_1-\xi_0\leq(\mu_{0,0}-1)\xi_0+\varepsilon_0$ and
\begin{align}
(\forall n\in\NN\smallsetminus\{0\})\quad\xi_{n+1}-\xi_n
&\leq (\mu_{n,n}-1)\xi_n+(1-\mu_{n,n})\xi_{n-1}
+\varepsilon_n\nonumber\\
&=(\mu_{n,n}-1)(\xi_n-\xi_{n-1})+\varepsilon_n,
\end{align}
that $(\forall n\in\NN)$
$\theta_{n+1}\leq (\mu_{n,n}-1)\theta_n+\varepsilon_n
=(1+\nu_{n})\theta_n+\varepsilon_n$.
Consequently, we derive from Lemma~\ref{l:gr} and 
\eqref{24g9h8Hkj2-09B} that $(\forall n\in\NN)$
$\theta_{n+1}\leq\sum_{k=0}^{n}\varepsilon_k\exp(\zeta_{n,k})$.
Using {\rm\cite[Theorem~141]{Knop54}}, this yields
\begin{equation}
\label{e:yqfKi}
\Sum_{n\in\NN}\theta_{n+1}\leq\Sum_{n\in\NN}
\Sum_{k=0}^{n}\varepsilon_k\exp(\zeta_{n,k}) 
=\Sum_{k\in\NN}\varepsilon_k\Sum_{n\geq k}\exp(\zeta_{n,k}) 
=\Sum_{k\in\NN}\varepsilon_k\chi_k.
\end{equation}
Now set $(\forall n\in\NN)$ $\omega_n=\xi_n-\sum_{k=0}^n\theta_k$. 
Since $\sum_{k\in\NN}\chi_k\varepsilon_k<\pinf$, we infer from 
\eqref{e:yqfKi} that $\sum_{n\in\NN}\theta_n<\pinf$. 
Thus, since $\inf_{n\in\NN}\xi_n\geq 0$, 
$(\omega_n)_{n\in\NN}$ is bounded below and 
\begin{equation}
(\forall n\in\NN)\quad
\omega_{n+1}=\xi_{n+1}-\theta_{n+1}-\sum_{k=0}^n\theta_k\leq
\xi_{n+1}-\xi_{n+1}+\xi_n-\sum_{k=0}^n\theta_k=\omega_n.
\end{equation}
Altogether, $(\omega_n)_{n\in\NN}$ converges, and so does
therefore $(\xi_n)_{n\in\NN}$.
\end{proof}

\section{Asymptotic behavior of Algorithm~\ref{algo:1}}
\label{sec:3}

The main result of the paper is the following theorem, which
analyzes the asymptotic behavior of Algorithm~\ref{algo:1}.

\begin{theorem}
\label{t:1}
Consider the setting of Algorithm~\ref{algo:1}.
For every $n\in\NN$, define
\begin{equation}
\label{e:roma2013-07-04o}
\vartheta_n=\lambda_n\sum_{i=1}^m\|e_{i,n}\|\quad\text{and}\quad
(\forall i\in\{1,\ldots,m\})\quad T_{i+,n}=
\begin{cases}
T_{i+1,n}\cdots T_{m,n},&\text{if}\;\;i\neq m;\\
\Id,&\text{if}\;\;i=m,
\end{cases}
\end{equation}
and set
\begin{equation} 
\label{Kj8Ygf2-21t}
\nu_n\colon S\to\RP\colon x\mapsto\vartheta_n
\big(2\|\overline{x}_n-x\|+\vartheta_n\big).
\end{equation}
Then the following hold:
\begin{enumerate}
\item
\label{t:1-i}
Let $n\in\NN$ and $x\in S$. Then
$\|x_{n+1}-x\|\leq
\sum_{j=0}^n|\mu_{n,j}|\,\|x_j-x\|+\vartheta_n$.
\item
\label{t:1-ii}
Let $n\in\NN$ and $x\in S$. Then
\begin{align*}
\|x_{n+1}-x\|^2
&\leq\sum_{j=0}^n\mu_{n,j}\|x_j-x\|^2
-\dfrac{1}{2}
\sum_{j=0}^n\Sum_{k=0}^n\mu_{n,j}\mu_{n,k}\|x_j-x_k\|^2\\
&\quad\;-\lambda_n(1/\phi_n-\lambda_n)
\|T_n\overline{x}_n-\overline{x}_n\|^2+\nu_n(x).
\end{align*}
\item
\label{t:1-iii}
Let $n\in\NN$ and $x\in S$. Then
\begin{align*}
\hskip-4mm \|x_{n+1}-x\|^2
&\leq\Sum_{j=0}^n\mu_{n,j}\|x_j-x\|^2-\dfrac{1}{2}
\Sum_{j=0}^n\Sum_{k=0}^n\mu_{n,j}\mu_{n,k}\|x_j-x_k\|^2\\
&\quad\;+\lambda_n(\lambda_n-1)
\|T_n\overline{x}_n-\overline{x}_n\|^2\\
&\quad\;-\lambda_n\underset{1\leq i\leq m}{\text{\rm max}}
\bigg(\dfrac{1-\alpha_{i,n}}{\alpha_{i,n}}
\left\|(\Id-T_{i,n})T_{i+,n}\overline{x}_n-
(\Id-T_{i,n})T_{i+,n}x\right\|^2\bigg)+\nu_n(x).
\end{align*}
\end{enumerate}
Now assume that, in addition, 
\begin{equation}
\label{e:1a}
\Sum_{n\in\NN}\chi_n\sum_{j=0}^n\sum_{k=0}^n
[\mu_{n,j}\mu_{n,k}]^-\|x_j-x_k\|^2<\pinf
\quad\text{and}\quad
(\forall x\in S)\quad\Sum_{n\in\NN}\chi_n\nu_n(x)<\pinf.
\end{equation}
Then the following hold:
\begin{enumerate}
\setcounter{enumi}{3}
\itemsep0mm 
\item
\label{t:1i}
Let $x\in S$. Then $(\|x_n-x\|)_{n\in\NN}$ converges.
\item
\label{t:1iva}
$\lambda_n(1/\phi_n-\lambda_n)
\|T_n\overline{x}_n-\overline{x}_n\|^2\to 0$.
\item
\label{t:1ivb}
$\sum_{j=0}^n\sum_{k=0}^n[\mu_{n,j}\mu_{n,k}]^+\|x_j-x_k\|^2\to 0$.
\item
\label{t:1ii}
Suppose that 
\begin{equation}
\label{24g9h8Hkj1-23}
(\exi\varepsilon\in\zeroun)(\forall n\in\NN)\quad 
\lambda_n\leq(1-\varepsilon)/\phi_n.
\end{equation}
Then $x_{n+1}-\overline{x}_n\to 0$. In addition, if 
every weak sequential cluster point of 
$(\overline{x}_n)_{n\in\NN}$ is in $S$, 
then there exists $x\in S$ such that $x_n\weakly x$.
\item
\label{t:1iii}
Suppose that $(\overline{x}_n)_{n\in\NN}$ has a strong cluster 
point $x$ in $S$ and that \eqref{24g9h8Hkj1-23} holds. Then 
$x_n\to x$.
\item
\label{t:1ivc}
Let $x\in S$ and suppose that 
$(\exi\varepsilon\in\zeroun)(\forall n\in\NN)$ 
$\lambda_n\leq\varepsilon+(1-\varepsilon)/\phi_n$. Then
\[
\lambda_n\underset{1\leq i\leq m}{\text{\rm max}}
\dfrac{1-\alpha_{i,n}}{\alpha_{i,n}}
\left\|(\Id-T_{i,n})T_{i+,n}\overline{x}_n-(\Id-T_{i,n})T_{i+,n}x
\right\|^2\to 0.
\]
\end{enumerate}
\end{theorem}
\begin{proof}
Let $n\in\NN$ and set
\begin{equation} 
\label{Kj8Ygf2-21r}
e_n=T_{1,n}\Big(T_{2,n}\big(\cdots 
T_{m-1,n}(T_{m,n}\overline{x}_n+e_{m,n})
+e_{m-1,n}\cdots\big)+e_{2,n}\Big)+e_{1,n}-T_n\overline{x}_n.
\end{equation}
If $m>1$, using the nonexpansiveness of the operators 
$(T_{i,n})_{1\leq i\leq m-1}$, we obtain
\begin{align}
\label{e:qf8}
\|e_n\|
&\leq\|e_{1,n}\|+\bigg\|T_{1,n}\bigg(T_{2,n}\big(\cdots
T_{m-1,n}(T_{m,n}\overline{x}_n
+e_{m,n})+e_{m-1,n}\cdots\big)+e_{2,n}\bigg)-T_{1,n}\cdots 
T_{m,n}\overline{x}_n\bigg\|\nonumber\\[2mm]
&\leq\|e_{1,n}\|+\nonumber\\[2mm]
&\quad\;\bigg\|T_{2,n}\bigg(T_{3,n}\big(\cdots T_{m-1,n}
(T_{m,n}\overline{x}_n+e_{m,n})+e_{m-1,n}\cdots\big)+e_{3,n}\bigg)
+e_{2,n}-T_{2,n}\cdots T_{m,n}\overline{x}_n\bigg\|\nonumber\\[2mm]
&\leq\|e_{1,n}\|+\|e_{2,n}\|+\nonumber\\[2mm]
&\quad\;\bigg\|T_{3,n}\bigg(T_{4,n}\big(\cdots T_{m-1,n}
(T_{m,n}\overline{x}_n+e_{m,n})+e_{m-1,n}\cdots\big)+e_{4,n}\bigg)
+e_{3,n}-T_{3,n}\cdots T_{m,n}\overline{x}_n\bigg\|\nonumber\\
&\;\;\vdots\nonumber\\
&\leq\sum_{i=1}^m\|e_{i,n}\|.
\end{align}
Thus, we infer from \eqref{e:roma2013-07-04o} that
\begin{equation}
\label{24g9h8Hkj2-12}
\lambda_n\|e_n\|\leq\vartheta_n.
\end{equation}
On the other hand, we derive from \eqref{Kj8Ygf2-21s} and 
\eqref{Kj8Ygf2-21r} that
\begin{equation}
\label{e:zn}
x_{n+1}=\overline{x}_n+\lambda_n
\big(T_n\overline{x}_n+e_n-\overline{x}_n\big).
\end{equation}
Now set
\begin{equation}
\label{e:roma2013-07-04u}
R_n=\dfrac{1}{\phi_n}T_n-\dfrac{1-\phi_n}{\phi_n}\Id
\quad\text{and}\quad\eta_n=\lambda_n\phi_n.
\end{equation}
Then $\eta_n\in\rzeroun$, $\Fix R_n=\Fix T_n$, and $R_n$ is 
quasinonexpansive since $T_n$ is averaged quasinonexpansive with
constant $\phi_n$ by Remark~\ref{r:t_n}\ref{r:t_nii}.
Furthermore, \eqref{e:zn} can be written as
\begin{equation}
\label{e:roma2013-07-04a}
x_{n+1}=\overline{x}_n+
\eta_n\big(R_n\overline{x}_n-\overline{x}_n\big)+\lambda_ne_n.
\end{equation}
Next, we define
\begin{equation}
\label{Kj8Ygf2-27x}
z_n=\overline{x}_n+\lambda_n(T_n\overline{x}_n-\overline{x}_n)
=\overline{x}_n+\eta_n(R_n\overline{x}_n-\overline{x}_n). 
\end{equation}
Let $x\in S$. Since $x\in\Fix R_n$ and $R_n$ is
quasinonexpansive, we have
\begin{align}
\label{e:roma2013-07-04t}
\|z_n-x\|
&=\|(1-\eta_n)(\overline{x}_n-x)+
\eta_n(R_n\overline{x}_n-x)\|\nonumber\\
&\leq(1-\eta_n)\|\overline{x}_n-x\|+\eta_n\|R_n\overline{x}_n-x\|
\nonumber\\
&\leq\|\overline{x}_n-x\|.
\end{align}
Hence, \eqref{e:roma2013-07-04a} and \eqref{24g9h8Hkj2-12} yield
\begin{equation}
\label{e:roma2013-07-04c}
\|x_{n+1}-x\|\leq\|z_n-x\|+\lambda_n\|e_n\|
\leq\|z_n-x\|+\vartheta_n.
\end{equation}
In turn, it follows from \eqref{e:roma2013-07-04t} and
\eqref{Kj8Ygf2-21t} that
\begin{equation}
\label{Kj8Ygf2-28e}
\|x_{n+1}-x\|^2
\leq\|z_n-x\|^2+2\vartheta_n\|z_n-x\|+\vartheta_n^2
\leq\|z_n-x\|^2+\nu_n(x).
\end{equation}
In addition, \cite[Lemma~2.14(ii)]{Livre1} yields
\begin{equation}
\label{Kj8Ygf2-20o}
\|\overline{x}_n-x\|^2=\bigg\|\sum_{j=0}^n\mu_{n,j}(x_j-x)\bigg\|^2
=\sum_{j=0}^n\mu_{n,j}\|x_j-x\|^2-\dfrac{1}{2}\sum_{j=0}^n
\sum_{k=0}^n\mu_{n,j}\mu_{n,k}\|x_j-x_k\|^2.
\end{equation}

\ref{t:1-i}: By \eqref{e:roma2013-07-04c}
and \eqref{e:roma2013-07-04t}, 
\begin{equation}
(\forall n\in\NN)(\forall x\in S)\quad 
\|x_{n+1}-x\|\leq\|\overline{x}_n-x\|+\vartheta_n
\leq\sum_{j=0}^n|\mu_{n,j}|\,\|x_j-x\|+\vartheta_n.
\end{equation}

\ref{t:1-ii}: 
Let $n\in\NN$ and $x\in S$. Since
\begin{align}
\label{Kj8Ygf2-27b}
\|z_n-x\|^2
&=\|(1-\eta_n)(\overline{x}_n-x)+
\eta_n(R_n\overline{x}_n-x)\|^2\nonumber\\
&=(1-\eta_n)\|\overline{x}_n-x\|^2+
\eta_n\|R_n\overline{x}_n-x\|^2-\eta_n(1-\eta_n)
\|R_n\overline{x}_n-\overline{x}_n\|^2\nonumber\\
&\leq\|\overline{x}_n-x\|^2-\eta_n(1-\eta_n)
\|R_n\overline{x}_n-\overline{x}_n\|^2,
\end{align}
we deduce from \eqref{Kj8Ygf2-28e} and \eqref{e:roma2013-07-04u} 
that
\begin{align}
\|x_{n+1}-x\|^2
&\leq\|z_n-x\|^2+\nu_n(x)\nonumber\\
&\leq\|\overline{x}_n-x\|^2-\eta_n(1-\eta_n)
\|R_n\overline{x}_n-\overline{x}_n\|^2+\nu_n(x)\nonumber\\
&=\|\overline{x}_n-x\|^2-
\lambda_n(1/\phi_n-\lambda_n)
\|T_n\overline{x}_n-\overline{x}_n\|^2+\nu_n(x).
\end{align}
In view of \eqref{Kj8Ygf2-20o}, 
we obtain the announced inequality.

\ref{t:1-iii}: 
Let $n\in\NN$ and $x\in S$.
We derive from \cite[Proposition~4.35]{Livre1} that
\begin{align}
\label{e:roma2013-07-05a}
(\forall i\in\{1,\ldots,m-1\})(\forall (u,v)\in\HH^2)\nonumber\\
\|T_{i,n}u-T_{i,n}v\|^2
&\leq\|u-v\|^2-\dfrac{1-\alpha_{i,n}}{\alpha_{i,n}}
\|(\Id-T_{i,n})u-(\Id-T_{i,n})v\|^2.
\end{align}
If $m>1$, using this inequality successively for $i=1,\ldots,m-1$
leads to 
\begin{align}
\label{e:roma2013-07-04q}
\|T_n\overline{x}_n-x\|^2
&=\left\|T_{1,n}\cdots T_{m,n}\overline{x}_n-T_{1,n}\cdots 
T_{m,n}x\right\|^2\nonumber\\
&\leq\|T_{m,n}\overline{x}_n-T_{m,n}x\|^2-\sum_{i=1}^{m-1}
\dfrac{1-\alpha_{i,n}}{\alpha_{i,n}}
\left\|(\Id-T_{i,n})T_{i+,n}\overline{x}_n-
(\Id-T_{i,n})T_{i+,n}x\right\|^2\nonumber\\
&\leq\|T_{m,n}\overline{x}_n-T_{m,n}x\|^2-\max_{1\leq i\leq m-1}
\dfrac{1-\alpha_{i,n}}{\alpha_{i,n}}
\left\|(\Id-T_{i,n})T_{i+,n}\overline{x}_n-
(\Id-T_{i,n})T_{i+,n}x\right\|^2.
\end{align}
Note that, in cases \ref{prob:1a} and \ref{prob:1c} of 
Problem~\ref{prob:1},
\begin{equation}
\label{Kj8Ygf2-28a}
\|T_{m,n}\overline{x}_n-T_{m,n}x\|^2\leq
\|\overline{x}_n-x\|^2-\dfrac{1-\alpha_{m,n}}{\alpha_{m,n}}
\|(\Id-T_{m,n})\overline{x}_n-(\Id-T_{m,n})x\|^2.
\end{equation}
This inequality remains valid in case \ref{prob:1b}
of Problem~\ref{prob:1} since \cite[Proposition~4.49(i)]{Livre1} 
implies that
\begin{equation}
\label{Kj8Ygf2-28c}
\Fix(T_{1,n}\cdots T_{m,n})=\bigcap_{i=1}^m\Fix T_{i,n}
\end{equation}
and, therefore, that $x\in\Fix T_{m,n}$. 
Altogether, we deduce from \eqref{e:roma2013-07-04q} and 
\eqref{Kj8Ygf2-28a} that
\begin{equation}
\label{Kj8Ygf2-28b}
\|T_n\overline{x}_n-x\|^2\leq
\|\overline{x}_n-x\|^2-\max_{1\leq i\leq m}
\dfrac{1-\alpha_{i,n}}{\alpha_{i,n}}
\left\|(\Id-T_{i,n})T_{i+,n}\overline{x}_n-
(\Id-T_{i,n})T_{i+,n}x\right\|^2.
\end{equation}
Hence, it follows from \eqref{Kj8Ygf2-27x} that
\begin{align}
\label{Kj8Ygf2-28d}
\|z_n-x\|^2
&=\|(1-\lambda_n)(\overline{x}_n-x)+
\lambda_n(T_n\overline{x}_n-x)\|^2\nonumber\\
&=(1-\lambda_n)\|\overline{x}_n-x\|^2+
\lambda_n\|T_n\overline{x}_n-x\|^2+\lambda_n(\lambda_n-1)
\|T_n\overline{x}_n-\overline{x}_n\|^2\nonumber\\
&\leq\|\overline{x}_n-x\|^2-\lambda_n\max_{1\leq i\leq m}
\dfrac{1-\alpha_{i,n}}{\alpha_{i,n}}
\left\|(\Id-T_{i,n})T_{i+,n}\overline{x}_n-
(\Id-T_{i,n})T_{i+,n}x\right\|^2\nonumber\\
&\quad\;+\lambda_n(\lambda_n-1)
\|T_n\overline{x}_n-\overline{x}_n\|^2.
\end{align}
In view of \eqref{Kj8Ygf2-28e} and \eqref{Kj8Ygf2-20o}, the 
inequality is established. 

\ref{t:1i}: Let $x\in S$ and set
\begin{equation}
(\forall n\in\NN)\quad
\begin{cases}
\xi_n=\|x_n-x\|^2\\
\varepsilon_n=\nu_n(x)+\dfrac{1}{2}\Sum_{j=0}^n
\Sum_{k=0}^n[\mu_{n,j}\mu_{n,k}]^-\|x_j-x_k\|^2.
\end{cases}
\end{equation}
Since $\inf_{n\in\NN}\lambda_n(1/\phi_n-\lambda_n)\geq 0$, 
\eqref{e:1a} and \ref{t:1-ii} imply that
\begin{equation}
\label{e:2-17-01-03a}
\sum_{n\in\NN}\chi_n\varepsilon_n<\pinf\quad\text{and}\quad
(\forall n\in\NN)\quad 
\xi_{n+1}\leq\sum_{j=0}^n\mu_{n,j}\xi_j+\varepsilon_n.
\end{equation} 
In turn, it follows from \eqref{e:1a} and condition 
\ref{algo:1e} in Algorithm~\ref{algo:1} that
$(\|x_n-x\|)_{n\in\NN}$ converges.

\ref{t:1iva}--\ref{t:1ivb}: 
Let $x\in S$. Then it follows from \ref{t:1i} that
$\rho=\lim_{n\to\pinf}\|x_n-x\|$ is well defined. Hence, 
Lemma~\ref{l:lem2} implies that
$\sum_{j=0}^n\mu_{n,j}\|x_j-x\|^2\to\rho^2$ and therefore that
\begin{equation}
\label{24g9h8Hkj1-08b}
\sum_{j=0}^n\mu_{n,j}\|x_j-x\|^2-\|x_{n+1}-x\|^2\to 0.
\end{equation}
Since $\inf_{n\in\NN}\chi_n>0$, \eqref{e:1a} yields 
\begin{equation}
\label{24g9h8Hkj1-14a}
\nu_n(x)\to 0\quad\text{and}\quad
\Sum_{j=0}^n\Sum_{k=0}^n[\mu_{n,j}\mu_{n,k}]^-\|x_j-x_k\|^2\to 0.
\end{equation}
It follows from \ref{t:1-ii}, \eqref{24g9h8Hkj1-08b}, and
\eqref{24g9h8Hkj1-14a} that 
\begin{align}
\label{24g9h8Hkj1-08c1}
0&\leq
\lambda_n(1/\phi_n-\lambda_n)\|T_n\overline{x}_n-\overline{x}_n\|^2
+\dfrac{1}{2}\Sum_{j=0}^n\Sum_{k=0}^n[\mu_{n,j}\mu_{n,k}]^+
\|x_j-x_k\|^2\nonumber\\
&\leq\Sum_{j=0}^n\mu_{n,j}\|x_j-x\|^2-\|x_{n+1}-x\|^2+\dfrac{1}{2}
\Sum_{j=0}^n\Sum_{k=0}^n[\mu_{n,j}\mu_{n,k}]^-\|x_j-x_k\|^2
+\nu_n(x)\nonumber\\
&\to 0,
\end{align}
which gives the desired conclusions.

\ref{t:1ii}: 
Set $\zeta=1/\varepsilon-1$.
We deduce from \eqref{e:1a} and \eqref{Kj8Ygf2-21t} 
that $\sum_{n\in\NN}\vartheta_n^2<\pinf$. Hence, it follows 
from \eqref{e:zn}, \eqref{24g9h8Hkj2-12}, \eqref{24g9h8Hkj1-23}, 
and \ref{t:1iva} that
\begin{align}
\label{24g9h8Hkj1-24a}
\|x_{n+1}-\overline{x}_n\|^2
&\leq 2\Big(\lambda_n^2
\big\|T_n\overline{x}_n-\overline{x}_n\big\|^2
+\lambda_n^2\|e_n\|^2\Big)\nonumber\\
&\leq 2\bigg(\dfrac{\lambda_n}{1/\phi_n-\lambda_n}
\lambda_n(1/\phi_n-\lambda_n)
\big\|T_n\overline{x}_n-\overline{x}_n\big\|^2
+\vartheta_n^2\bigg)\nonumber\\
&\leq2\Big(\zeta\lambda_n(1/\phi_n-\lambda_n)
\big\|T_n\overline{x}_n-\overline{x}_n\big\|^2
+\vartheta_n^2\Big)\nonumber\\
&\to 0.
\end{align}
Therefore $x_{n+1}-\overline{x}_n\weakly 0$ and hence the weak 
sequential cluster points of $(x_n)_{n\in\NN}$ lie in $S$. In
view of \ref{t:1i} and \cite[Lemma~2.47]{Livre1}, the claim 
is proved.

\ref{t:1iii}: Since $x_{n+1}-\overline{x}_n\to 0$ by 
\eqref{24g9h8Hkj1-24a}, $(x_n)_{n\in\NN}$ 
has a strong cluster point $x\in S$. In view of \ref{t:1i}, 
$x_n\to x$.

\ref{t:1ivc}: Set $\zeta=1/\varepsilon-1$. Then, for every 
$n\in\NN$, $\lambda_n\leq 1/(1+\zeta)+\zeta/(\phi_n(1+\zeta))$ and 
therefore $(1+\zeta)\lambda_n-1\leq\zeta/\phi_n$, i.e., 
$\lambda_n-1\leq\zeta(1/\phi_n-\lambda_n)$.
We therefore derive from \ref{t:1-iii}, 
\eqref{24g9h8Hkj1-08b}, \eqref{e:1a}, \ref{t:1iva},
and \eqref{24g9h8Hkj1-14a} that 
\begin{align}
\label{24g9h8Hkj1-08c2}
0&\leq
\lambda_n\underset{1\leq i\leq m}{\text{\rm max}}
\dfrac{1-\alpha_{i,n}}{\alpha_{i,n}}
\left\|(\Id-T_{i,n})T_{i+,n}\overline{x}_n-(\Id-T_{i,n})
T_{i+,n}x\right\|^2\nonumber\\
&\leq\Sum_{j=0}^n\mu_{n,j}\|x_j-x\|^2-\|x_{n+1}-x\|^2+\dfrac{1}{2}
\Sum_{j=0}^n\Sum_{k=0}^n[\mu_{n,j}\mu_{n,k}]^-\|x_j-x_k\|^2
\nonumber\\
&\quad\;+\lambda_n(\lambda_n-1)
\|T_n\overline{x}_n-\overline{x}_n\|^2+\nu_n(x)\nonumber\\
&\leq\Sum_{j=0}^n\mu_{n,j}\|x_j-x\|^2-\|x_{n+1}-x\|^2+\dfrac{1}{2}
\Sum_{j=0}^n\Sum_{k=0}^n[\mu_{n,j}\mu_{n,k}]^-\|x_j-x_k\|^2
\nonumber\\
&\quad\;+\zeta\lambda_n(1/\phi_n-\lambda_n)
\|T_n\overline{x}_n-\overline{x}_n\|^2+\nu_n(x)\nonumber\\
&\to 0,
\end{align}
which shows the assertion.
\end{proof}

Next, we present two corollaries that are instrumental in the
analysis of two important special cases of our framework: mean 
value and inertial multi-layer algorithms. 

\begin{corollary}
\label{c:1}
Consider the setting of Algorithm~\ref{algo:1} and define
$(\vartheta_n)_{n\in\NN}$ as in \eqref{e:roma2013-07-04o}. 
Assume that
\begin{equation}
\label{e:1f}
\inf_{n\in\NN}\min_{0\leq j\leq n}\mu_{n,j}\geq 0\quad
\text{and}\quad
\Sum_{n\in\NN}\chi_n\vartheta_n<\pinf.
\end{equation}
Then the following hold:
\begin{enumerate}
\item
\label{c:1i}
$\sum_{j=0}^n\sum_{k=0}^n\mu_{n,j}\mu_{n,k}\|x_j-x_k\|^2\to 0$.
\item
\label{c:1ii}
Let $x\in S$ and suppose that 
$(\exi\varepsilon\in\zeroun)(\forall n\in\NN)$ 
$\lambda_n\leq\varepsilon+(1-\varepsilon)/\phi_n$. Then
\[
\lambda_n\underset{1\leq i\leq m}{\text{\rm max}}
\dfrac{1-\alpha_{i,n}}{\alpha_{i,n}}
\left\|(\Id-T_{i,n})T_{i+,n}\overline{x}_n-(\Id-T_{i,n})T_{i+,n}x
\right\|^2\to 0.
\]
\item
\label{c:1iii}
Suppose that every weak sequential cluster point of 
$(\overline{x}_n)_{n\in\NN}$ is in $S$ and that 
$(\exi\varepsilon\in\zeroun)(\forall n\in\NN)$
$\lambda_n\leq(1-\varepsilon)/\phi_n$. Then 
$x_{n+1}-\overline{x}_n\to 0$ and there exists $x\in S$ 
such that $x_n\weakly x$.
\item
\label{c:1iii'}
Suppose that every weak sequential 
cluster point of $(\overline{x}_n)_{n\in\NN}$ is in $S$, 
that $\sup_{n\in\NN}\phi_n<1$, and that
$(\exi\varepsilon\in\zeroun)(\forall n\in\NN)$ 
$\lambda_n\leq\varepsilon+(1-\varepsilon)/\phi_n$. 
Then $x_{n+1}-\overline{x}_n\to 0$ and there exists $x\in S$ 
such that $x_n\weakly x$.
\item
\label{c:1iv}
Suppose that every weak sequential cluster point of 
$(\overline{x}_n)_{n\in\NN}$ is in $S$ and that 
$\inf_{n\in\NN}\mu_{n,n}>0$. Then 
$x_{n}-\overline{x}_n\to 0$ and there exists $x\in S$ 
such that $x_n\weakly x$.
\end{enumerate}
\end{corollary}
\begin{proof}
We derive from Theorem~\ref{t:1}\ref{t:1-i} that
$(\forall n\in\NN)$
$\|x_{n+1}-x\|\leq\sum_{j=0}^n\mu_{n,j}\|x_j-x\|+\vartheta_n$.
In turn, it follows from condition \ref{algo:1e} in 
Algorithm~\ref{algo:1} that 
$(\|x_n-x\|)_{n\in\NN}$ converges.
As a result, $(x_n)_{n\in\NN}$ is bounded and 
\eqref{e:1f} therefore implies \eqref{e:1a}. 

\ref{c:1i}--\ref{c:1iii}: These follow respectively from items
\ref{t:1ivb}, \ref{t:1ivc}, and \ref{t:1ii} in Theorem~\ref{t:1}.

\ref{c:1iii'}:
Set $\delta=\varepsilon(1-\sup_{n\in\NN}\phi_n)$.
Then $\delta\in\left]0,\varepsilon\right[$ and
$(\forall n\in\NN)$ $(\varepsilon-\delta)/\phi_n\geq\varepsilon$.
Hence, 
\begin{equation}
(\forall n\in\NN)\quad\lambda_n
\leq\varepsilon+\frac{1-\varepsilon}{\phi_n}
=\varepsilon+\dfrac{1-\delta}{\phi_n}
-\dfrac{\varepsilon-\delta}{\phi_n}
\leq\dfrac{1-\delta}{\phi_n}. 
\end{equation}
The claim therefore follows from \ref{c:1iii}. 

\ref{c:1iv}: Set
\begin{equation}
\label{24g9h8Hkj2-08d}
\theta=\dfrac{1}{\displaystyle{\inf_{n\in\NN}}\:\mu_{n,n}}
\quad\text{and}\quad
(\forall n\in\NN)(\forall j\in\{0,\ldots,n\})\quad
\gamma_{n,j}=
\begin{cases}
\dfrac{\mu_{n,n}+1}{2},&\text{if}\;\;j=n;\\[3mm]
\dfrac{\mu_{n,j}}{2},&\text{if}\;\;j<n.
\end{cases}
\end{equation}
Then, using Apollonius' identity, \cite[Lemma~2.12(iv)]{Livre1}, 
and \eqref{24g9h8Hkj2-08d}, we obtain
\begin{align}
\label{e:rdu-cdg}
\dfrac{1}{4}\|\overline{x}_n-x_n\|^2
&=\dfrac{1}{2}\Big(\|\overline{x}_n-x\|^2+\|x_n-x\|^2\Big)-
\bigg\|\dfrac{\overline{x}_n+x_n}{2}-x\bigg\|^2\nonumber\\
&=\dfrac{1}{2}\Big(\|\overline{x}_n-x\|^2+\|x_n-x\|^2\Big)
-\bigg\|\sum_{j=0}^n\gamma_{n,j}(x_j-x)\bigg\|^2\nonumber\\
&=\dfrac{1}{2}\Big(\|\overline{x}_n-x\|^2+\|x_n-x\|^2\Big)
-\sum_{j=0}^n\gamma_{n,j}\|x_j-x\|^2+\sum_{0\leq j<k\leq n}
\gamma_{n,j}\gamma_{n,k}\|x_j-x_k\|^2\nonumber\\
&\leq\dfrac{1}{2}\Bigg(\sum_{j=0}^n\mu_{n,j}\|x_j-x\|^2
+\|x_n-x\|^2\Bigg)-\sum_{j=0}^n\gamma_{n,j}\|x_j-x\|^2\nonumber\\
&\quad\;+\dfrac{1}{4}\Bigg(\sum_{0\leq j<k< n}\mu_{n,j}\mu_{n,k}
\|x_j-x_k\|^2+\sum_{j=0}^{n-1}\mu_{n,j}(\mu_{n,n}+1)
\|x_j-x_n\|^2\Bigg)\nonumber\\
&\leq\dfrac{1}{2}\Bigg(\sum_{j=0}^n\mu_{n,j}\|x_j-x\|^2
+\|x_n-x\|^2\Bigg)-\sum_{j=0}^n\gamma_{n,j}\|x_j-x\|^2\nonumber\\
&\quad\;+\dfrac{1}{4}\Bigg(\sum_{0\leq j<k\leq n}\mu_{n,j}\mu_{n,k}
\|x_j-x_k\|^2+\theta\sum_{j=0}^{n-1}\mu_{n,j}\mu_{n,n}
\|x_j-x_n\|^2\Bigg)\nonumber\\
&\leq\dfrac{1}{2}\Bigg(\sum_{j=0}^n\mu_{n,j}\|x_j-x\|^2
+\|x_n-x\|^2\Bigg)-\sum_{j=0}^n\gamma_{n,j}\|x_j-x\|^2\nonumber\\
&\quad\;+\dfrac{1+\theta}{4}\sum_{0\leq j<k\leq n}\mu_{n,j}\mu_{n,k}
\|x_j-x_k\|^2.
\end{align}
Next, let us set $\rho=\lim\|x_n-x\|^2$. Then it follows from
Lemma~\ref{l:lem2} that $\sum_{j=0}^n\mu_{n,j}\|x_j-x\|^2\to\rho$ 
and $\sum_{j=0}^n\gamma_{n,j}\|x_j-x\|^2\to\rho$. 
On the other hand, \ref{c:1i} asserts that
$\sum_{0\leq j<k\leq n}\mu_{n,j}\mu_{n,k}\|x_j-x_k\|^2\to 0$.
Altogether, \eqref{e:rdu-cdg} yields
$\|\overline{x}_n-x_n\|\to 0$. 
Thus, the weak sequential cluster points of $(x_n)_{n\in\NN}$ 
belong to $S$, and the conclusion follows from 
the fact that $(\|x_n-x\|)_{n\in\NN}$ converges
and \cite[Lemma~2.47]{Livre1}.
\end{proof}

\begin{corollary}
\label{c:2}
Consider the setting of Algorithm~\ref{algo:1} with 
$(\forall i\in\{1,\ldots,m\})(\forall n\in\NN)$ $ e_{i,n}=0$.
Set $x_{-1}=x_0$ and suppose that there exists a sequence 
$(\eta_n)_{n\in\NN}$ in $\lzeroun$ such that $\eta_0=0$ and
\begin{equation}
\label{e:matinertiel}
(\forall n\in\NN)(\forall j\in\{0,\ldots,n\})\quad\mu_{n,j}=
\begin{cases}
1+\eta_n,&\text{if}\;\:j=n;\\
-\eta_n,&\text{if}\;\:j=n-1;\\
0,&\text{if}\;\:j<n-1.
\end{cases}
\end{equation}
For every $n\in\NN$, set
\begin{equation}
\label{24g9h8Hkj2-09D}
(\forall k\in\NN)\quad\zeta_{k,n}=
\begin{cases}
0,&\text{if}\;\:k\leq n;\\
\Sum_{j=n+1}^{k}(\eta_j-1),&\text{if}\;\:k>n,
\end{cases}
\end{equation}
and assume that
$\chi_n=\sum_{k\geq n}\exp(\zeta_{k,n})$.
Suppose that one of the following is satisfied:
\begin{enumerate}[label=\rm(\alph*)]
\itemsep0mm 
\item 
\label{c:2a}
$\sum_{n\in\NN}\chi_n\eta_n\|x_n-x_{n-1}\|^2<\pinf$.
\item 
\label{c:2b} 
$\sum_{n\in\NN} n\|x_n-x_{n-1}\|^2<\pinf$ 
and there exists $\tau\in\left[2,\pinf\right[$ such that 
$(\forall n\in\NN\smallsetminus\{0\})$ $\eta_n=(n-1)/(n+\tau)$. 
\item 
\label{c:2c}
$\sum_{n\in\NN}\eta_n\|x_n-x_{n-1}\|^2<\pinf$ 
and there exists $\eta\in\left[0,1\right[$ such that 
$(\forall n\in\NN)$ $\eta_n\leq\eta$.
\item 
\label{c:2d} 
Set $(\forall n\in\NN)$ $\omega_n=1/\phi_n-\lambda_n$.
There exist $(\sigma,\vartheta)\in\left]0,\pinf\right[^2$
and $\eta\in\left]0,1\right[$ such that
\begin{equation}
(\forall n\in\NN)\quad
\begin{cases}
\eta_{n}\leq\eta_{n+1}\leq\eta\\
\lambda_n\leq 
\dfrac{\vartheta/\phi_n-\eta\big(\eta(1+\eta)
+\eta\vartheta\omega_{n+1}+\sigma\big)}
{\vartheta\big(1+\eta(1+\eta)+
\eta\vartheta\omega_{n+1}+\sigma\big)}\\
\dfrac{\eta^2(1+\eta)+\eta\sigma}{\vartheta}
<\dfrac{1}{\phi_n}-\eta^2\omega_{n+1}.
\end{cases}
\end{equation}
\end{enumerate}
Then the following hold:
\begin{enumerate}
\itemsep0mm 
\item
\label{c:2i0}
$\lambda_n(1/\phi_n-\lambda_n)
\|T_n\overline{x}_n-\overline{x}_n\|^2\to 0$.
\item
\label{c:2i}
Let $x\in S$ and suppose that 
$(\exi\varepsilon\in\zeroun)(\forall n\in\NN)$ 
$\lambda_n\leq\varepsilon+(1-\varepsilon)/\phi_n$. Then
\[
\lambda_n\underset{1\leq i\leq m}{\text{\rm max}}
\dfrac{1-\alpha_{i,n}}{\alpha_{i,n}}
\left\|(\Id-T_{i,n})T_{i+,n}\overline{x}_n-(\Id-T_{i,n})T_{i+,n}x
\right\|^2\to 0.
\]
\item
\label{c:2ii0}
$\overline{x}_n-x_n\to 0$.
\item
\label{c:2ii}
Suppose that every weak sequential cluster point of 
$(\overline{x}_n)_{n\in\NN}$ is in $S$.
Then there exists $x\in S$ such that $x_n\weakly x$.
\end{enumerate}
\end{corollary}
\begin{proof}
In view of Example~\ref{ex:main},
$(\mu_{n,j})_{n\in\NN, 0\leq j\leq n}$ satisfies conditions
\ref{algo:1a}--\ref{algo:1e} in Algorithm~\ref{algo:1}. 

\ref{c:2a}: 
Set $\chi=\inf_{n\in\NN}\chi_n$ and define $(\nu_n)_{n\in\NN}$ 
as in \eqref{Kj8Ygf2-21t}. 
We have $\sup_{n\in\NN} (1+\eta_n)\leq 2$ and
\begin{equation}
(\forall n\in\NN)\quad
\begin{cases}
\chi_n\sum_{j=0}^n\sum_{k=0}^n
[\mu_{n,j}\mu_{n,k}]^-\|x_j-x_k\|^2=
(1+\eta_n)\chi_n\eta_n\|x_n-x_{n-1}\|̂^2\\
(\forall x\in S)\quad\chi_n\nu_n(x)=0.
\end{cases}
\end{equation}
Hence \eqref{e:1a} holds, and \ref{c:2i0} and \ref{c:2i} follow from 
Theorem~\ref{t:1}\ref{t:1iva}\&\ref{t:1ivc}, respectively. 
Furthermore, \eqref{e:matinertiel} implies that
\begin{equation}
\|\overline{x}_n-x_n\|^2\leq 
\eta_n^2\|x_n-x_{n-1}\|^2
\leq\eta_n\|x_n-x_{n-1}\|^2
\leq\frac{\chi_n\eta_n}{\chi}\|x_n-x_{n-1}\|^2\to 0.
\end{equation}
Thus, \ref{c:2ii0} holds. In turn, the weak sequential
cluster points of $(x_n)_{n\in\NN}$ belong to $S$ and 
\ref{c:2ii} therefore follows from Theorem~\ref{t:1}\ref{t:1i} and
\cite[Lemma~2.47]{Livre1}.

\ref{c:2b}$\Rightarrow$\ref{c:2a}: 
It follows from Lemma~\ref{l:chi}\ref{l:chi_i} that
\begin{equation}
\Sum_{n\in\NN}\chi_n\eta_n\|x_n-x_{n-1}\|^2\leq
\Sum_{n\in\NN}\dfrac{n+7}{2}\|x_n-x_{n-1}\|^2<\pinf.
\end{equation}

\ref{c:2c}$\Rightarrow$\ref{c:2a}: 
Lemma~\ref{l:chi}\ref{l:chi_ii} asserts that
$\sup_{n\in\NN}\chi_n\leq e/(1-\eta)$.

\ref{c:2d}$\Rightarrow$\ref{c:2c}:
Let $x\in S$. It follows from Theorem~\ref{t:1}\ref{t:1-ii} that
\begin{align}
(\forall n\in\NN)\quad\|x_{n+1}-x\|^2\leq &(1+\eta_n)\|x_n-
x\|^2-\eta_n\|x_{n-1}-
x\|^2+\eta_n(1+\eta_n)\|x_{n}-x_{n-1}\|^2\nonumber\\
\label{e:cin1}
&-\lambda_n\bigg(\frac{1}{\phi_n}-\lambda_n\bigg)
\|T_n\overline{x}_n-\overline{x}_n\|^2.
\end{align}
Now set $\beta_{-1}=\|x_{0}-x\|^2$ and 
\begin{equation}
\label{e:defpdp}
(\forall n\in\NN)\quad\beta_n=\|x_{n}-x\|^2,\quad
\delta_n=\|x_{n}-x_{n-1}\|^2,\quad\text{and}\quad
\rho_n=\dfrac{1}{\eta_n+\lambda_n\vartheta}.
\end{equation}
Then
\begin{align}
(\forall n\in\NN)\quad\|T_n\overline{x}_n-\overline{x}_n\|^2
&=\dfrac{1}{\lambda_n^2}\big\|(x_{n+1}-x_n)+\eta_n (x_{n-1}-x_n)
\big\|^2\nonumber\\
&=\dfrac{1}{\lambda_n^2}\bigg(\delta_{n+1}+\eta_n^2\delta_{n}
+\eta_n\bigg(2\Scal{\sqrt{\rho_n}(x_{n+1}-x_n)}
{\frac{x_{n-1}-x_n}{\sqrt{\rho_n}}}\bigg)\bigg)\nonumber\\
&\geq\dfrac{1}{\lambda_n^2}\bigg(\delta_{n+1}+
\eta_n^2\delta_{n}-\eta_n\bigg(\rho_n\delta_{n+1}+
\dfrac{\delta_n}{\rho_n}\bigg)\bigg).
\end{align}
Thus, we derive from \eqref{e:cin1} that
\begin{equation} 
(\forall n\in\NN)\quad
\beta_{n+1}-\beta_n-\eta_n(\beta_n-\beta_{n-1})
\leq\dfrac{(1/\phi_n-\lambda_n)(\eta_n\rho_n-1)}
{\lambda_n}\delta_{n+1}+\gamma_n\delta_n,
\end{equation}
where 
\begin{equation}
\label{e:cingamman}
(\forall n\in\NN)\quad
\gamma_n=\eta_n(1+\eta_n)+\eta_n
\bigg(\dfrac{1}{\phi_n}-\lambda_n\bigg)
\dfrac{1-\rho_n\eta_n}{\rho_n\lambda_n}\geq 0.
\end{equation}
However, it follows from \eqref{e:defpdp} that
$(\forall n\in\NN)$ $\vartheta=
(1-\rho_n\eta_n)/(\rho_n\lambda_n)$.
Hence, \eqref{e:cingamman} yields
\begin{equation}
(\forall n\in\NN)\quad\gamma_{n}=\eta_{n}(1+\eta_{n})+
\eta_{n}\bigg(\dfrac{1}{\phi_{n}}-\lambda_{n}\bigg)\vartheta 
\leq\eta(1+\eta)+\eta\vartheta\omega_{n}.
\end{equation}
Thus, by Lemma~\ref{l:1}, 
$\sum_{n\in\NN}\eta_n\delta_n\leq\sum_{n\in\NN}\delta_n<\pinf$ 
and we conclude that \ref{c:2c} is satisfied.
\end{proof}

\begin{remark}
\label{r:solsbury}
In Corollary~\ref{c:2}, no errors were allowed in the
implementation of the operators. It is however possible to allow
errors in multi-layer inertial methods in certain scenarios. For
instance, suppose that in Corollary~\ref{c:2} we make the
additional assumptions that $\lambda_n\equiv 1$ and that
$\bigcup_{n\in\NN}\ran T_{1,n}$ is bounded. At the same time,
let us introduce errors of such that 
$(\forall i\in\{1,\ldots,m\})$
$\sum_{n\in\NN}\chi_n\|e_{i,n}\|<\pinf$. Note that 
\eqref{Kj8Ygf2-21s} becomes
\begin{equation}
\label{Kj8Ygf2-21S}
\begin{array}{l}
\text{for}\;n=0,1,\ldots\\
\left\lfloor
\begin{array}{l}
\overline{x}_n=(1+\eta_n)x_n-\eta_nx_{n-1}\\
x_{n+1}=T_{1,n}\Big(T_{2,n}
\big(\cdots T_{m-1,n}(T_{m,n}\overline{x}_n\!+\!e_{m,n})
\!+\!e_{m-1,n}\cdots\big)\!+\!e_{2,n}\Big)
\!+\!e_{1,n}.
\end{array}
\right.\\
\end{array}
\end{equation}
Hence, the assumptions imply that $(x_n)_{n\in\NN}$ is bounded.
In turn, $(\overline{x}_n)_{n\in\NN}$ is bounded and it follows from
\eqref{Kj8Ygf2-21t} that $(\forall x\in S)$
$\sum_{n\in\NN}\chi_n\nu_n(x)<\pinf$. An inspection of the proof of
Corollary~\ref{c:2} then reveals immediately that its conclusions
under any of assumptions \ref{c:2a}--\ref{c:2c} remain true.
\end{remark}

\section{Examples and Applications}
\label{sec:4}

In this section we exhibit various existing results as special
cases of our framework. Our purpose is not to exploit it to its 
full capacity but rather to illustrate its potential on simple 
instances. We first recover the main result of \cite{Jmaa02} on 
algorithm \eqref{e:3}.

\begin{example}
\label{ex:jmaa02}
We consider the setting studied in \cite{Jmaa02}.
Let $(T_n)_{n\in\NN}$ be a sequence of firmly quasinonexpansive
operators from $\HH$ to $\HH$ such that 
$S=\bigcap_{n\in\NN}\Fix T_n\neq\emp$. Then the problem of finding 
a point in $S$ is a special case of 
Problem~\ref{prob:1}\ref{prob:1c} 
where we assume that $\alpha_{1,n}\equiv 1/2$. In addition,
let $(e_n)_{n\in\NN}$ be a sequence in $\HH$ such that
$\sum_{n\in\NN}\|e_n\|<\pinf$, let
$(\lambda_n)_{n\in\NN}$ be a sequence in $\left]0,2\right[$ such
that $0<\inf_{n\in\NN}\lambda_n\leq\sup_{n\in\NN}\lambda_n<2$,
and let $(\mu_{n,j})_{n\in\NN, 0\leq j\leq n}$ be an array with
entries in $\RP$ which satisfies the following:
\begin{enumerate}[label=\rm(\alph*)]
\itemsep0mm 
\item
\label{a:7a}
$(\forall n\in\NN)$ $\sum_{j=0}^n\mu_{n,j}=1$.
\item
\label{a:7b}
$(\forall j\in\NN)$
$\lim\limits_{\substack{n\to\pinf}}\mu_{n,j}=0$.
\item
\label{a:7c}
Every sequence $(\xi_n)_{n\in\NN}$ in $\RP$ such that
\begin{equation}
\Big(\exi(\varepsilon_n)_{n\in\NN}\in\RP^\NN\Big)\:\;
\begin{cases}
\sum_{n\in\NN}\varepsilon_n<\pinf\\
(\forall n\in\NN)\quad 
\xi_{n+1}\leq\sum_{j=0}^n\mu_{n,j}\xi_j+\varepsilon_n
\end{cases}
\end{equation} 
converges.
\end{enumerate}
Clearly, conditions \ref{a:7a}--\ref{a:7c} above imply that, in 
Algorithm~\ref{algo:1}, conditions \ref{algo:1a}--\ref{algo:1e} 
are satisfied. Now let $x_0\in\HH$, and define a 
sequence $(x_n)_{n\in\NN}$ by
\begin{equation}
\label{e:w3}
(\forall n\in\NN)\quad
x_{n+1}=\overline{x}_n+\lambda_n\big(T_n\overline{x}_n+e_n-
\overline{x}_n\big),
\;\;\text{where}\;\;
\overline{x}_n=\Sum_{j=0}^n\mu_{n,j}x_j,
\end{equation}
which corresponds to a 1-layer instance of \eqref{Kj8Ygf2-21s}.
This mean iteration process was seen in \cite{Jmaa02} to cover
several classical mean iteration methods, as well as memoryless
convex feasibility algorithms \cite{Else01} (see also
\cite{Borw17}).  The result obtained in
\cite[Theorem~3.5(i)]{Jmaa02} on the weak convergence of
$(x_n)_{n\in\NN}$ to a point in $S$ corresponds to the special
case of Corollary~\ref{c:1}\ref{c:1iii} in which we further set
$\chi_{n}\equiv 1$. 
\end{example}

Next, we retrieve the main result of \cite{Jmaa15} on the
convergence of an overrelaxed version of \eqref{e:6} and the
special cases discussed there, in particular those of \cite{Opti04}. 

\begin{example}
\label{ex:jmaa15}
We consider the setting studied in \cite{Jmaa15}, which
corresponds to Problem~\ref{prob:1}\ref{prob:1a}. Given
$x_0\in\HH$ and sequences $(e_{1,n})_{n\in\NN}$, \ldots,
$(e_{m,n})_{n\in\NN}$ in $\HH$ such that 
$\sum_{n\in\NN}\lambda_n\sum_{i=1}^m\|e_{i,n}\|<\pinf$,
construct a sequence $(x_n)_{n\in\NN}$ via the $m$-layer recursion
\begin{multline}
\label{e:w6}
(\forall n\in\NN)\quad
x_{n+1}=x_n+\lambda_n\Big(T_{1,n}\Big(T_{2,n}
\big(\cdots T_{m-1,n}(T_{m,n}{x}_n+e_{m,n})
+e_{m-1,n}\cdots\big)+e_{2,n}\Big)+e_{1,n}-{x}_n\Big),\\
\quad\text{where}\quad
0<\lambda_n\leq\varepsilon+(1-\varepsilon)/\phi_n.
\end{multline}
Note that \eqref{e:w6} corresponds the memoryless version of
\eqref{Kj8Ygf2-21s}. The result on the weak convergence of 
$(x_n)_{n\in\NN}$ obtained in 
\cite[Theorem~3.5(iii)]{Jmaa15} corresponds to the special case 
of Corollary~\ref{c:1}\ref{c:1iii'} in which the following
additional assumptions are made: 
\begin{enumerate}[label=\rm(\alph*)]
\itemsep0mm 
\item
\label{ex:jmaa15a}
$(\forall n\in\NN)$ $\chi_n=1$ and 
$(\forall j\in\{0,\ldots,n\})$ $\mu_{n,j}=\delta_{n,j}$.
\item
\label{ex:jmaa15c}
$(\exi\varepsilon\in\zeroun)(\forall n\in\NN)$ $\lambda_n
\leq(1-\varepsilon)(1/\phi_n+\varepsilon)$.
\end{enumerate}
Note that condition \ref{ex:jmaa15a} above 
implies that, in Algorithm~\ref{algo:1}, conditions 
\ref{algo:1a}--\ref{algo:1c} trivially hold, while condition 
\ref{algo:1e} follows from \cite[Lemma~5.31]{Livre1}. 
We also observe that \cite[Theorem~3.5(iii)]{Jmaa15} itself 
extends the results of \cite[Section~3]{Opti04}, where the
relaxation parameters $(\lambda_n)_{n\in\NN}$ are confined 
to $\rzeroun$.
\end{example}

The next two examples feature mean value and inertial iterations
in the case of a single quasinonexpansive operator. As is easily
seen, the memoryless algorithm \eqref{e:1} can fail to produce a
convergent sequence in this scenario. 

\begin{example}
\label{ex:Mann}
Let $T\colon\HH\to\HH$ be a quasinonexpansive operator such that 
$\Id-T$ is demiclosed at $0$ and $\Fix T\neq\emp$,
let $(\mu_{n,j})_{n\in\NN, 0\leq j\leq n}$ be an array in $\RP$ that
satisfies conditions \ref{algo:1a}--\ref{algo:1e} in 
Algorithm~\ref{algo:1} with $\chi_n\equiv 1$ and such that
$\inf_{n\in\NN}\;\mu_{n+1,n}\mu_{n+1,n+1}>0$,
and let $(e_n)_{n\in\NN}$ be a sequence in $\HH$ such that
$\sum_{n\in\NN}\|e_n\|<\pinf$. Let $x_0\in\HH$ and iterate
\begin{equation}
\label{ex:Mann1}
(\forall n\in\NN)\quad x_{n+1}=T\overline{x}_n+e_n,
\quad\text{where}\quad\overline{x}_n=\Sum_{j=0}^n\mu_{n,j}x_j.
\end{equation}
Then $T\overline{x}_n-\overline{x}_n\to 0$ and 
there exists $x\in\Fix T$ such that $x_n\weakly x$ and
$\overline{x}_n\weakly x$. 
\end{example}
\begin{proof}
We apply Corollary~\ref{c:1} in the setting of
Problem~\ref{prob:1}\ref{prob:1c} with $T_{1,n}\equiv T$,
$\alpha_{1,n}\equiv 1$, $\phi_{n}\equiv 1$, and
$\lambda_n\equiv 1$. 
First, note that \eqref{e:1f} is satisfied. Furthermore,
Corollary~\ref{c:1}\ref{c:1iv} entails that 
$\overline{x}_n-x_n\to 0$, while 
Corollary~\ref{c:1}\ref{c:1i} yields 
$\mu_{n+1,n}\mu_{n+1,n+1}\|x_{n+1}-x_n\|^2\to 0$ and hence
$x_{n+1}-x_n\to 0$. Therefore
$\overline{x}_n-T\overline{x}_n=
(\overline{x}_n-x_n)+(x_n-x_{n+1})+e_n\to 0$. 
Since $\Id-T$ is demiclosed at $0$, it follows that
every weak sequential cluster point of 
$(\overline{x}_n)_{n\in\NN}$ is in $\Fix T$. In view of 
Corollary~\ref{c:1}\ref{c:1iv}, the proof is complete. 
\end{proof}

\begin{example}
\label{ex:botcsetnek}
Let $T\colon\HH\to\HH$ be a quasinonexpansive operator such that 
$\Id-T$ is demiclosed at $0$ and $\Fix T\neq\emp$,
and let $(\eta_n)_{n\in\NN}$ be a sequence 
in $\lzeroun$ such that $\eta_0=0$, 
$\eta=\sup_{n\in\NN}\eta_n<1$, and 
$(\forall n\in\NN)$ $\eta_n\leq\eta_{n+1}$.
Let $(\sigma,\vartheta)\in\RPP^2$ be such that
$(\eta^2(1+\eta)+\eta\sigma)/\vartheta<1-\eta^2$,
and let $(\lambda_n)_{n\in\NN}$ be a
sequence in $\zeroun$ such that 
$0<\inf_{n\in\NN}\lambda_n\leq\sup_{n\in\NN}\lambda_n\leq 
(\vartheta-\eta(\eta(1+\eta)
+\eta\vartheta+\sigma))/
(\vartheta(1+\eta(1+\eta)+\eta\vartheta+\sigma))$.
Let $x_0\in\HH$, set $x_{-1}=x_0$, and iterate
\begin{equation}
(\forall n\in\NN)\quad x_{n+1}=\overline{x}_n+
\lambda_n\big(T\overline{x}_n 
-\overline{x}_n\big),
\quad\text{where}\quad
\overline{x}_n=(1+\eta_n)x_n-\eta_nx_{n-1}. 
\end{equation}
Then $T\overline{x}_n-\overline{x}_{n}\to 0$ and there exists
$x\in\Fix T$ such that $x_n\weakly x$. In the case when $T$ is 
nonexpansive, this result appears in 
\cite[Theorem~5]{Botc15}.
\end{example}
\begin{proof}
This is an instance of 
Corollary~\ref{c:2}\ref{c:2d}\ref{c:2i0}\&\ref{c:2ii} 
and Problem~\ref{prob:1}\ref{prob:1c} in which
$T_{1,n}\equiv T$,
$\alpha_{1,n}\equiv 1$, and $\phi_{n}\equiv 1$.
Note that condition \ref{c:2d} in Corollary~\ref{c:2} 
is satisfied since $(\forall n\in\NN)$ $\omega_n=1-\lambda_n<1$.
\end{proof}

Next, we consider applications to monotone operator splitting.
Let us recall basic notions about a set-valued operator
$A\colon\HH\to 2^{\HH}$ \cite{Livre1}.
We denote by $\ran A=\menge{u\in\HH}{(\exi x\in\HH)\;u\in Ax}$ 
the range of $A$, by 
$\dom A=\menge{x\in\HH}{Ax\neq\emp}$ the domain of $A$, by 
$\zer A=\menge{x\in\HH}{0\in Ax}$ the set of zeros of $A$, 
by $\gra A=\menge{(x,u)\in\HH\times\HH}{u\in Ax}$ the graph of 
$A$, and by $A^{-1}$ the inverse of $A$, i.e., the operator 
with graph
$\menge{(u,x)\in\HH\times\HH}{u\in Ax}$. The resolvent of $A$ is
$J_A=(\Id+A)^{-1}$ and $s\colon\dom A\to\HH$ is a
selection of $A$ if $(\forall x\in\dom A)$ $s(x)\in Ax$.
Moreover, $A$ is monotone if
\begin{equation}
(\forall (x,u)\in\gra A)(\forall (y,v)\in\gra A)\quad
\scal{x-y}{u-v}\geq 0,
\end{equation}
and maximally monotone if there exists no monotone operator 
$B\colon\HH\to 2^{\HH}$ such that $\gra A\subset\gra B\neq\gra A$.
In this case, $J_A$ is a firmly nonexpansive operator defined 
everywhere on $\HH$ and the reflector $R_A=2J_A-\Id$ is nonexpansive.
We denote by $\Gamma_0(\HH)$ the class of proper lower semicontinuous
convex functions from $\HH$ to $\RX$.
Let $f\in\Gamma_0(\HH)$. For every $x\in\HH$, $f+\|x-\cdot\|^2/2$
possesses a unique minimizer, which is denoted by $\prox_f x$.
We have $\prox_f=J_{\partial f}$, where
\begin{equation}
\partial f\colon\HH\to 2^\HH\colon x\mapsto
\menge{u\in\HH}{(\forall y\in\HH)\; 
\scal{y-x}{u}+f(x)\leq f(y)}
\end{equation}
is the Moreau subdifferential of $f$.
Our convergence results will rest on the following asymptotic
principle.

\begin{lemma}
\label{l:op1}
Let $A$ and $B$ be maximally monotone operators from $\HH$ to
$2^\HH$, let $(x_n,u_n)_{n\in\NN}$ be a sequence in $\gra A$,
let $(y_n,v_n)_{n\in\NN}$ be a sequence in $\gra B$, 
let $x\in\HH$, and let $v\in\HH$. Suppose that $x_n\weakly x$,
$v_n\weakly v$, $x_n-y_n\to 0$, and $u_n+v_n\to 0$.
Then the following hold:
\begin{enumerate}
\itemsep0mm 
\item 
\label{l:op1i}
$(x,-v)\in\gra A$ and $(x,v)\in\gra B$.
\item 
\label{l:op1ii}
$0\in Ax+Bx$ and $0\in-A^{-1}(-v)+B^{-1}v$.
\end{enumerate}
\end{lemma}
\begin{proof}
Apply \cite[Proposition~26.5]{Livre1} with $\mathcal{K}=\HH$ and 
$L=\Id$.
\end{proof}

As discussed in \cite{Opti04}, many splitting methods can
be analyzed within the powerful framework of fixed point methods
for averaged operators. The analysis provided in the present paper
therefore makes it possible to develop new methods in this
framework, for instance mean value or inertial versions of 
standard splitting methods. We provide two such examples below. 
First, we consider the
Peaceman-Rachford splitting method, which typically does not 
converge unless strong requirements are imposed on the underlying
operators \cite{Joca09}. In the spirit of Mann's work 
\cite{Mann53}, we show that mean iterations induce the convergence 
of this algorithm. 

\begin{proposition}
\label{p:dr}
Let $A\colon\HH\to2^\HH$ and $B\colon\HH\to 2^\HH$ be maximally
monotone operators such that $\zer(A+B)\neq\emp$ and 
let $\gamma\in\left]0,\pinf\right[$. Let
$(a_n)_{n\in\NN}$ and $(b_n)_{n\in\NN}$ be sequences in $\HH$
such that $\sum_{n\in\NN}\|a_n\|<\pinf$
and $\sum_{n\in\NN}\|b_n\|<\pinf$, let $x_0\in\HH$, and let
$(\mu_{n,j})_{n\in\NN, 0\leq j\leq n}$ be an array in $\RP$ that
satisfies conditions \ref{algo:1a}--\ref{algo:1e} in 
Algorithm~\ref{algo:1} with $\chi_n\equiv 1$ and such that
$\inf_{n\in\NN}\;\mu_{n+1,n}\mu_{n+1,n+1}>0$. Iterate
\begin{equation}
\label{e:pr1}
\begin{array}{l}
\text{for}\;n=0,1,\ldots\\
\left\lfloor
\begin{array}{l}
\overline{x}_n=\Sum_{j=0}^n\mu_{n,j}x_j\\
y_n=J_{\gamma B}\overline{x}_n+b_n\\
z_n=J_{\gamma A}(2y_n-\overline{x}_n)+a_n\\
x_{n+1}=\overline{x}_n+2(z_n-y_n).
\end{array}
\right.\\
\end{array}
\end{equation}
Then there exists $x\in\Fix R_{\gamma A}R_{\gamma B}$ 
such that $x_n\weakly x$ and $\overline{x}_n\weakly x$.  
Now set $y=J_{\gamma B}x$. 
Then $y\in\zer(A+B)$, $z_n-y_n\to 0$, 
$y_n\weakly y$, and $z_n\weakly y$.
\end{proposition}
\begin{proof}
Set $T=R_{\gamma A}R_{\gamma B}$ and 
$(\forall n\in\NN)$
$e_n=2a_n+R_{\gamma A}(R_{\gamma B}\overline{x}_n+2b_n) 
-R_{\gamma A}(R_{\gamma B}\overline{x}_n)$. 
Then $T$ is nonexpansive, $\Id-T$ is therefore demiclosed,
and, since $\zer(A+B)\neq\emp$, 
\cite[Proposition~26.1(iii)(b)]{Livre1} yields
$\Fix T=\Fix R_{\gamma A}R_{\gamma B}\neq\emp$.
In addition, we derive from \eqref{e:pr1} that
\begin{equation}
\label{e:prT}
(\forall n\in\NN)\quad x_{n+1}=T\overline{x}_n+e_n,
\end{equation}
where 
\begin{align}
\sum_{n\in\NN}\|e_n\|
&\leq\sum_{n\in\NN}\big(2\|a_n\|+\|R_{\gamma A}(R_{\gamma B}
\overline{x}_n+2b_n)-R_{\gamma A}(R_{\gamma B}\overline{x}_n)\|\big)
\nonumber\\
&\leq\sum_{n\in\NN}2\big(\|a_n\|+\|b_n\|\big)
\nonumber\\
&<\pinf.
\end{align}
Consequently, we deduce from Example~\ref{ex:Mann} that 
$(x_n)_{n\in\NN}$ 
and $(\overline{x}_n)_{n\in\NN}$ converge weakly to a point 
$x\in\Fix T=\Fix R_{\gamma A}R_{\gamma B}$, and that
$T\overline{x}_n-\overline{x}_n\to 0$. In addition, 
\cite[Proposition~26.1(iii)(b)]{Livre1} asserts that
$y\in\zer(A+B)$. Next, we derive from \eqref{e:pr1} and 
\eqref{e:prT} that
$2(z_n-y_n)=x_{n+1}-\overline{x}_n=(T\overline{x}_n-\overline{x}_n)
+e_n\to 0$.
It remains to show that $y_n\weakly y$. 
Since $(\overline{x}_n)_{n\in\NN}$ converges weakly, it is
bounded. However, $(\forall n\in\NN)$ $\|y_n-y_0\|= 
\|J_{\gamma B}\overline{x}_n-J_{\gamma B} x_0+b_n\| 
\leq\|\overline{x}_n-x_0\|+\|b_n\|$. Therefore 
$(y_n)_{n\in\NN}$ is bounded. Now let $z$ be a 
weak sequential cluster point of 
$(y_n)_{n\in\NN}$, say $y_{k_n}\weakly z$. In view of 
\cite[Lemma~2.46]{Livre1}, it is enough to show that $z=y$.
To this end, set 
$(\forall n\in\NN)$ 
$v_n=\gamma^{-1}(\overline{x}_n-y_n+b_n)$
and
$w_n=\gamma^{-1}(2y_n-\overline{x}_n-z_n+a_n)$.
Then $(\forall n\in\NN)$ $(z_n-a_n,w_n)\in\gra A$ and
$(y_n-b_n,v_n)\in\gra B$.
In addition, we have $\overline{x}_{k_n}\weakly x$, 
$z_{k_n}-a_{k_n}\weakly z$, 
$v_{k_n}\weakly\gamma^{-1}(x-z)$,
$(z_{k_n}-a_{k_n})-(y_{k_n}-b_{k_n})\to 0$, and 
$v_{k_n}+w_{k_n}=\gamma^{-1}(y_{k_n}-z_{k_n}+a_{k_n}+b_{k_n})\to 0$.
Hence, we derive from Lemma~\ref{l:op1}\ref{l:op1i} that
$(z,\gamma^{-1}(x-z))\in\gra B$, i.e., $z=J_{\gamma B}x=y$.
\end{proof}

\begin{remark}
Let $f$ and $g$ be functions in $\Gamma_0(\HH)$, and specialize
Proposition~\ref{p:dr} to $A=\partial f$ and $B=\partial g$.
Then $\zer(A+B)=\Argmin(f+g)$. Moreover, \eqref{e:pr1} becomes
\begin{equation}
\label{e:fb12}
\begin{array}{l}
\text{for}\;n=0,1,\ldots\\
\left\lfloor
\begin{array}{l}
\overline{x}_n=\Sum_{j=0}^n\mu_{n,j}x_j\\
y_n=\prox_{\gamma g}\overline{x}_n+b_n\\
z_n=\prox_{\gamma f}(2y_n-\overline{x}_n)+a_n\\
x_{n+1}=\overline{x}_n+2(z_n-y_n),
\end{array}
\right.\\
\end{array}
\end{equation}
and we conclude that there exists a point $y\in\Argmin(f+g)$ such
that $y_n\weakly y$ and $z_n\weakly y$.
\end{remark}

We now propose a new forward-backward splitting framework which
includes existing instances as special cases. 
The following notion will be needed to establish
strong convergence properties (see \cite[Proposition~2.4]{Sico10}
for special cases). 

\begin{definition}{\rm \cite[Definition~2.3]{Sico10}}
\label{d:demir}
An operator $A\colon\HH\to 2^{\HH}$ is \emph{demiregular} at
$x\in\dom A$ if, for every sequence $(x_n,u_n)_{n\in\NN}$ in 
$\gra A$ and every $u\in Ax$ such that $x_n\weakly x$ and 
$u_n\to u$, we have $x_n\to x$.
\end{definition}

\begin{proposition}
\label{p:fb}
Let $\beta\in\left]0,\pinf\right[\,$, let 
$\varepsilon\in\left]0,\min\{1/2,\beta\}\right[\,$, 
let $x_0\in\HH$,
let $A\colon\HH\to 2^{\HH}$ be maximally monotone, 
and let $B\colon\HH\to\HH$ be $\beta$-cocoercive, i.e.,
\begin{equation}
(\forall x\in\HH)(\forall y\in\HH)\quad
\scal{x-y}{Bx-By}/\beta\geq\|Bx-By\|^2.
\end{equation}
Furthermore, let $(\gamma_n)_{n\in\NN}$ be a sequence in
$\left[\varepsilon,2\beta/(1+\varepsilon)\right]\,$, let
$(\mu_{n,j})_{n\in\NN, 0\leq j\leq n}$ be a real array that
satisfies conditions \ref{algo:1a}--\ref{algo:1e} in 
Algorithm~\ref{algo:1}, and
let $(a_n)_{n\in\NN}$ and $(b_n)_{n\in\NN}$ 
be sequences in $\HH$ such that 
$\sum_{n\in\NN}\chi_n\|a_n\|<\pinf$ 
and $\sum_{n\in\NN}\chi_n\|b_n\|<\pinf$. 
Suppose that $\zer(A+B)\neq\emp$ and that
\begin{equation}
\label{e:fb1}
(\forall n\in\NN)\quad
\lambda_n\in\left[\varepsilon,1+(1-\varepsilon)
\bigg(1-\dfrac{\gamma_n}{2\beta}\bigg)\right], 
\end{equation}
and set $(\forall n\in\NN)$ 
$\phi_n=2/(4-\gamma_n/\beta)$.
For every $n\in\NN$, iterate
\begin{equation}
\label{e:fb2}
x_{n+1}=\overline{x}_n+\lambda_n\Big(J_{\gamma_n A}
\big(\overline{x}_n-\gamma_n(B\overline{x}_n+b_n)\big)
+a_n-\overline{x}_n\Big),\quad\text{where}\quad
\overline{x}_n=\Sum_{j=0}^n\mu_{n,j}x_j.
\end{equation}
Suppose that one of the following is satisfied:
\begin{enumerate}[label=\rm(\alph*)]
\setlength{\itemsep}{1pt}
\item 
\label{p:fba}
$\displaystyle{\inf_{n\in\NN}}\;
\displaystyle{\min_{0\leq j\leq n}}$ $\mu_{n,j}\geq 0$.
\item 
\label{p:fbc}
$a_n\equiv b_n\equiv 0$, $(\mu_{n,j})_{n\in\NN, 0\leq j\leq n}$
satisfies \eqref{e:matinertiel}, and one of conditions
\ref{c:2a}--\ref{c:2d} in Corollary~\ref{c:2} is satisfied.
\end{enumerate}
Then the following hold:
\begin{enumerate}
\itemsep0mm 
\item
\label{p:fbi}
$J_{\gamma_n A}(\overline{x}_n-\gamma_nB\overline{x}_n)
-\overline{x}_n\to 0$.
\item
\label{p:fbii}
Let $z\in\zer(A+B)$. Then $B\overline{x}_n\to Bz$.
\item
\label{p:fbiii}
There exists $x\in\zer(A+B)$ such that $x_n\weakly x$.
\item
\label{p:fbiv}
Suppose that $A$ or $B$ is demiregular at every point in 
$\zer(A+B)$. Then there exists $x\in\zer(A+B)$ such 
that $x_n\to x$.
\end{enumerate}
\end{proposition}
\begin{proof}
We apply Corollary~\ref{c:1} in case \ref{p:fba}
and from Corollary~\ref{c:2} in case \ref{p:fbc}.
We first note that \eqref{e:fb2} is an instance of
Algorithm~\ref{algo:1} with $m=2$ and $(\forall n\in\NN)$
$T_{1,n}=J_{\gamma_n A}$, $T_{2,n}=\Id-\gamma_n B$,
$e_{1,n}=a_n$, and $e_{2,n}=-\gamma_nb_n$.
Indeed, for every $n\in\NN$, $T_{1,n}$ is 
$\alpha_{1,n}$-averaged with $\alpha_{1,n}=1/2$ 
\cite[Remark~4.34(iii) and Corollary~23.9]{Livre1},
$T_{2,n}$ is $\alpha_{2,n}$-averaged with 
$\alpha_{2,n}=\gamma_n/(2\beta)$ \cite[Proposition~4.39]{Livre1},
and the averaging constant of $T_{1,n}T_{2,n}$ is therefore given
by \eqref{e:phi_n} as 
\begin{equation}
\label{e:9gh}
\dfrac{{\alpha_{1,n}+\alpha_{2,n}-
2\alpha_{1,n}\alpha_{2,n}}}{1-\alpha_{1,n}\alpha_{2,n}}=
\dfrac{2}{4-\gamma_n/\beta}=\phi_n.
\end{equation}
On the other hand, we are in the setting of 
Problem~\ref{prob:1}\ref{prob:1a} since
\cite[Proposition~26.1(iv)(a)]{Livre1} yields
$(\forall n\in\NN)$
$\Fix(T_{1,n} T_{2,n})=\zer(A+B)\neq\emp$.
We also observe that, in view of \eqref{e:fb1},
\begin{equation}
\label{e:fb5}
(\forall n\in\NN)\quad\varepsilon\leq
\lambda_n\leq\varepsilon+\frac{1-\varepsilon}{\phi_n},\quad
\frac{1-\alpha_{1,n}}{\alpha_{1,n}}=1,\quad
\text{and}\quad\frac{1-\alpha_{2,n}}{\alpha_{2,n}}\geq\varepsilon
\end{equation}
which, by \eqref{e:phi_n}, yields
\begin{equation}
\sup_{n\in\NN}\phi_n\leq\dfrac{1+\varepsilon}{1+2\varepsilon}<1.
\end{equation}
In addition, it results from \eqref{e:9gh} that
$(\forall n\in\NN)$ $\lambda_n\leq 1/\phi_n+\varepsilon
\leq 2-\gamma_n/(2\beta)+\varepsilon\leq 2+\varepsilon$.
Therefore, 
\begin{equation}
\label{e:fb7}
\begin{cases}
\sum_{n\in\NN}\chi_n\lambda_n\|e_{1,n}\|=(2+\varepsilon)
\sum_{n\in\NN}\chi_n\|a_n\|<\pinf\\
\sum_{n\in\NN}\chi_n\lambda_n\|e_{2,n}\|\leq
2\beta(2+\varepsilon)\sum_{n\in\NN}\chi_n\|b_n\|<\pinf, 
\end{cases}
\end{equation}
which establishes \eqref{e:1f} for case \ref{p:fba}.
Altogether, \eqref{e:fb5}, Corollary~\ref{c:1}\ref{c:1ii}, and
Corollary~\ref{c:2}\ref{c:2i} imply that, for every
$z\in\zer(A+B)$, 
\begin{equation}
\label{e:fb8}
\begin{cases}
(T_{1,n}-\Id)T_{2,n}\overline{x}_n+(T_{2,n}-\Id)z=
(T_{1,n}-\Id)T_{2,n}\overline{x}_n-(T_{1,n}-\Id)T_{2,n}z
\to 0\\
(T_{2,n}-\Id)\overline{x}_n-(T_{2,n}-\Id)z\to 0.
\end{cases}
\end{equation}
Now set
\begin{equation}
\label{e:fb10}
(\forall n\in\NN)\quad y_n=J_{\gamma_n A}
(\overline{x}_n-\gamma_nB\overline{x}_n),\quad
u_n=\dfrac{\overline{x}_n-y_n}{\gamma_n}-
B\overline{x}_n,\quad
\text{and}\quad v_n=B\overline{x}_n,
\end{equation}
and note that 
\begin{equation}
\label{e:fb11}
(\forall n\in\NN)\quad u_n\in Ay_n.
\end{equation}

\ref{p:fbi}: Let $z\in\zer(A+B)$. Then \eqref{e:fb8} yields
$J_{\gamma_n A}(\overline{x}_n-\gamma_nB\overline{x}_n)
-\overline{x}_n=(T_{1,n}-\Id)T_{2,n}\overline{x}_n+
(T_{2,n}-\Id)z+(T_{2,n}-\Id)\overline{x}_n-(T_{2,n}-\Id)z\to 0$.

\ref{p:fbii}: 
We derive from \eqref{e:fb8} that
\begin{align} 
\|B\overline{x}_n-Bz\|
=\gamma_n^{-1}\|T_{2,n}\overline{x}_n-\overline{x}_n-T_{2,n}z+z\|
\leq\varepsilon^{-1}\|T_{2,n}\overline{x}_n-
\overline{x}_n-T_{2,n}z+z\|\to 0.
\end{align}

\ref{p:fbiii}: 
Let $(k_n)_{n\in\NN}$ be a strictly
increasing sequence in $\NN$ and let $y\in\HH$ be such that 
$\overline{x}_{k_n}\weakly y$. 
In view of Corollary~\ref{c:1}\ref{c:1iii'} in case \ref{p:fba},
and Corollary~\ref{c:2}\ref{c:2ii} in case \ref{p:fbc}, 
it remains to show that $y\in\zer(A+B)$. 
We derive from \ref{p:fbi} that $y_n-\overline{x}_n\to 0$.
Hence $y_{k_n}\weakly y$. Now let $z\in\zer(A+B)$. Then 
\ref{p:fbii} implies that $B\overline{x}_{n}\to Bz$. 
Altogether, $y_{k_n}\weakly y$, $v_{k_n}\weakly Bz$,
$y_{k_n}-\overline{x}_{k_n}\to 0$,
$u_{k_n}+v_{k_n}\to 0$, and, for every $n\in\NN$,
$u_{k_n}\in Ay_{k_n}$ and
$v_{k_n}\in B\overline{x}_{k_n}$.
It therefore follows from Lemma~\ref{l:op1}\ref{l:op1ii} that
$y\in\zer(A+B)$.

\ref{p:fbiv}: 
By \ref{p:fbiii}, there exists $x\in\zer(A+B)$ such that 
$x_n\weakly x$. In addition, we derive from
Corollary~\ref{c:1}\ref{c:1iii'} or 
Corollary~\ref{c:2}\ref{c:2ii0} that 
\begin{equation}
\label{e:fb77}
\overline{x}_n-x_{n+1}\to 0 
\quad\text{or}\quad
\overline{x}_n-x_n\to 0.
\end{equation}
Hence it follows from 
\eqref{e:fb10} and \ref{p:fbi} that $y_n\weakly x$,
and, from \eqref{e:fb10} and \ref{p:fbii}, that
$u_n\to -Bx$.
In turn, if $A$ is demiregular on $\zer(A+B)$, 
we derive from \eqref{e:fb11} that $y_n\to x$. Since 
$y_n-\overline{x}_n\to 0$, \eqref{e:fb77} yields $x_n\to x$. 
Now suppose that $B$ is demiregular on $\zer(A+B)$.
Since $\overline{x}_n\weakly x$, \ref{p:fbii} implies that
$\overline{x}_n\to x$ and it follows from \eqref{e:fb77}
that $x_n\to x$.
\end{proof}

\begin{remark}
\label{r:solsburyh}
As noted in Remark~\ref{r:solsbury}, we can allow errors in 
inertial multi-layer methods and, in particular, in the
inertial forward-backward algorithm. Thus, suppose that, in
Proposition~\ref{p:fb},
$\lambda_n\equiv 1$ and $A$ has bounded domain. Then
$\bigcup_{n\in\NN}\ran
T_{1,n}=\bigcup_{n\in\NN}\ran(\Id+\gamma_nA)^{-1}=\dom A$ is
bounded. Hence, it follows from Remark~\ref{r:solsbury} that, if
$\sum_{n\in\NN}\chi_n\|a_n\|<\pinf$ and
$\sum_{n\in\NN}\chi_n\|b_n\|<\pinf$, the conclusions of 
Proposition~\ref{p:fb}\ref{p:fbc} under any of 
assumptions \ref{c:2a}--\ref{c:2c} of Corollary~\ref{c:2}
remain true for the inertial forward-backward algorithm
\begin{equation}
\label{e:fbZ}
\begin{array}{l}
\text{for}\;n=0,1,\ldots\\
\left\lfloor
\begin{array}{l}
\overline{x}_n=(1+\eta_n)x_n-\eta_nx_{n-1}\\
x_{n+1}=J_{\gamma_n A}\big(\overline{x}_n-
\gamma_n(B\overline{x}_n+b_n)\big)+a_n.
\end{array}
\right.\\
\end{array}
\end{equation}
\end{remark}

\begin{remark}
\label{r:fbvar}
Let $f\in\Gamma_0(\HH)$, let $g\colon\HH\to\RR$ be convex and
differentiable with a $1/\beta$-Lipschitzian gradient, 
and suppose that $\Argmin(f+g)\neq\emp$. Then $\nabla g$ is
$\beta$-cocoercive \cite[Corollary~18.17]{Livre1}. Upon setting
$A=\partial f$ and $B=\nabla g$ in Proposition~\ref{p:fb}, we see
that, for every $n\in\NN$, \eqref{e:fb2} becomes
\begin{equation}
\label{e:fb13}
x_{n+1}=\overline{x}_n+\lambda_n\Big(\prox_{\gamma_n f}
\big(\overline{x}_n-\gamma_n(\nabla g(\overline{x}_n)+b_n)\big)
+a_n-\overline{x}_n\Big),\quad\text{where}\quad
\overline{x}_n=\Sum_{j=0}^n\mu_{n,j}x_j,
\end{equation}
and we conclude that there exists $x\in\Argmin(f+g)$ such that
$x_n\weakly x$ and $\nabla g(\overline{x}_n)\to\nabla g(x)$.
\end{remark}

\begin{remark}
Various results on the convergence of the forward-backward
splitting algorithm can be recovered from Proposition~\ref{p:fb}.
\begin{enumerate}
\itemsep0mm 
\item 
Suppose that $(\forall n\in\NN)$ $\lambda_n\leq 1$ and 
$(\forall j\in\{0,\ldots,n\})$ $\mu_{n,j}=\delta_{n,j}$.
Then conditions \ref{algo:1a}--\ref{algo:1e} in 
Algorithm~\ref{algo:1} hold with $\chi_n\equiv 1$
by \cite[Lemma~5.31]{Livre1}. In turn, 
Proposition~\ref{p:fb}\ref{p:fba}\ref{p:fbiii} 
reduces to \cite[Corollary~6.5]{Opti04}. In the context
of Remark~\ref{r:fbvar}, 
Proposition~\ref{p:fb}\ref{p:fba}\ref{p:fbiii}
captures \cite[Theorem~3.4(i)]{Smms05}. In this setting,
under suitable conditions on the errors, it is shown in 
\cite[Theorem~3(vi)]{Mapr17} that 
$(f+g)(x_n)-\min(f+g)(\HH)=o(1/n)$.
\item 
In the context of Remark~\ref{r:fbvar},
let $(\eta_n)_{n\in\NN}$ be a sequence in $\lzeroun$ that
satisfies condition \ref{c:2b} in Corollary~\ref{c:2},
and set $x_{-1}=x_0$ and $\eta_0=0$. If 
$\gamma_n\equiv\gamma_0\leq\beta$ and $\lambda_n\equiv 1$, 
Proposition~\ref{p:fb}\ref{p:fbc}\ref{p:fbiii}
covers the scheme 
\begin{equation}
(\forall n\in\NN)\quad
x_{n+1}=\prox_{\gamma_0 f}\Big(x_n+\eta_n(x_n-x_{n-1})-\gamma_0
\nabla g\big(x_n+\eta_n(x_n-x_{n-1})\big)\Big) 
\end{equation}
studied in \cite[Theorem~4.1]{Cham15}, where it was established that
$\sum_{n\in\NN} n\|x_n-x_{n-1}\|^2<\pinf$. In this case, it is
shown in \cite[Theorem~1]{Atto16} that
$(f+g)(x_n)-\min(f+g)(\HH)=o(1/n^2)$ .
\item  
If $\lambda_n\equiv 1$, 
then Proposition~\ref{p:fb}\ref{p:fbc}\ref{p:fbiii} 
under hypothesis \ref{c:2c} of Corollary~\ref{c:2} establishes a
statement made in \cite[Theorem~1]{Lore14}. Let us note, however, 
that the proof of \cite{Lore14} is not convincing as the authors 
appear to use the weak continuity of some operators which 
are merely strongly continuous.
\item 
Suppose that $(\forall n\in\NN)$
$\lambda_n\leq(1-\varepsilon)(2+\varepsilon-\gamma_n/(2\beta))$ 
and $(\forall j\in\{0,\ldots,n\})$ $\mu_{n,j}=\delta_{n,j}$. 
Then items \ref{p:fba}\ref{p:fbiii}
and \ref{p:fba}\ref{p:fbiv}
of Proposition~\ref{p:fb} capture respectively items
(iii) and (iv)(a)\&(b) of \cite[Proposition~4.4]{Jmaa15}.
In addition, in the context of Remark~\ref{r:fbvar},
Proposition~\ref{p:fb}\ref{p:fba}\ref{p:fbiii} captures 
\cite[Proposition~4.7(iii)]{Jmaa15}.
\item 
Proposition~\ref{p:fb} also applies to the proximal point
algorithm. Indeed, when $B=0$, it suffices to allow
$\beta=\pinf$ and $\alpha_{2,n}\equiv 0$, to set 
$1/\pinf=0$ and $1/0=\pinf$, and to take $(\gamma_n)_{n\in\NN}$ in
$\left[\varepsilon,\pinf\right[$.
In this setting, the proof remains valid and:
\begin{enumerate}
\item
Proposition~\ref{p:fb}\ref{p:fbc}\ref{p:fbi}\&%
\ref{p:fbiii}
under hypothesis \ref{c:2c} of Corollary~\ref{c:2} capture the
error-free case of 
\cite[Theorem~3.1]{Alva04}, while Theorem~\ref{t:1} covers its
general case. 
\item 
Let $\eta\in\left]0,1/3\right[$, set $\sigma=(1-3\eta)/2$
and $\vartheta=2/3$, and suppose that $\lambda_n\equiv 1$.
Then Proposition~\ref{p:fb}\ref{p:fbc} under hypothesis \ref{c:2d}
of Corollary~\ref{c:2} yields \cite[Proposition~2.1]{Alva01}.
\end{enumerate}
\end{enumerate}
\end{remark}

Next, we derive from Corollary~\ref{c:1} a mean value extension of 
Polyak's subgradient projection method \cite{Poly69} (likewise,
an inertial version can be derived from Corollary~\ref{c:2}).

\begin{example}
\label{ex:boris}
Let $C$ be a nonempty closed convex subset of $\HH$ with projector 
$P_C$, let $f\colon\HH\to\RR$ be a continuous convex function 
such that $\Argmin_C f\neq\emp$ and $\theta=\min f(C)$
is known. Suppose that one of the following holds:
\begin{enumerate}
\itemsep0mm 
\item
\label{ex:borisi}
$f$ is bounded on every bounded subset of $\HH$.
\item
\label{ex:borisii}
The conjugate $f^*$ of $f$ is supercoercive, i.e., 
$\lim_{\|u\|\to\pinf}f^*(u)/\|u\|=\pinf$.
\item
\label{ex:borisiii}
$\HH$ is finite-dimensional.
\end{enumerate}
Let $\eta\in\zeroun$, let
$\varepsilon\in\left]0,\eta/(2+\eta)\right[$, let 
$(\mu_{n,j})_{n\in\NN, 0\leq j\leq n}$ 
be an array in $\RP$ that satisfies conditions 
\ref{algo:1a}--\ref{algo:1e} in Algorithm~\ref{algo:1},
let $(\xi_n)_{n\in\NN}$ be in 
$[\eta,2-\eta]$, let $(\lambda_n)_{n\in\NN}$ be 
in $[\varepsilon,(1-\varepsilon)(2-\xi_n/2)]$,
let $s$ be a selection of $\partial f$, and let $x_0\in C$.
Iterate
\begin{equation}
\label{e:newcoldwar}
\begin{array}{l}
\text{for}\;n=0,1,\ldots\\
\left\lfloor
\begin{array}{l}
\overline{x}_n=\Sum_{j=0}^n\mu_{n,j}x_j\\
x_{n+1}=
\begin{cases}
\overline{x}_n+\lambda_n\bigg(P_C\bigg(\overline{x}_n+\xi_n
\Frac{\theta-f(\overline{x}_n)}
{\|s(\overline{x}_n)\|^2}s(\overline{x}_n)\bigg)
-\overline{x}_n\bigg),
&\text{if}\;\: s(\overline{x}_n)\neq 0;\\
\overline{x}_n,&\text{if}\;\:s(\overline{x}_n)=0.
\end{cases}
\end{array}
\right.
\end{array}
\end{equation}
Then there exists $x\in\Argmin_C f$ such that $x_n\weakly x$.
\end{example}
\begin{proof}
Let $G$ be the subgradient projector onto 
$D=\menge{x\in\HH}{f(x)\leq\theta}$ associated with $s$, that is,
\begin{equation}
G\colon\HH\to\HH\colon x\mapsto
\begin{cases}
x+\Frac{\theta-f(x)}{\|s(x)\|^2}s(x),&\text{if}\;\;f(x)>\theta;\\
x,&\text{if}\;\;f(x)\leq\theta.
\end{cases}
\end{equation}
Then $G$ is firmly quasinonexpansive 
\cite[Proposition~29.41(iii)]{Livre1}.
Now set $(\forall n\in\NN)$ $T_{1,n}=P_C$, 
$T_{2,n}=\Id+\xi_n(G-\Id)$, $\alpha_{1,n}=1/2$, and 
$\alpha_{2,n}=\xi_n/2$. Then, for every $n\in\HH$, 
$T_{1,n}$ is an $\alpha_{1,n}$-averaged nonexpansive operator 
\cite[Proposition~4.16]{Livre1}, $T_{2,n}$ is an
$\alpha_{2,n}$-averaged quasinonexpansive operator,
\cite[Proposition~4.49(i)]{Livre1} yields
\begin{equation}
\label{dFse7hY03a}
\Fix T_{1,n}T_{2,n}=\Fix T_{1,n}\cap\Fix T_{2,n}=
C\cap D=\Argmin_C f,
\end{equation}
and Remark~\ref{r:t_n}\ref{r:t_nii} asserts that 
$T_{1,n}T_{2,n}$ is an averaged quasinonexpansive operator 
with constant $\phi_n=2/(4-\xi_n)$ and
$\lambda_n\leq(1-\varepsilon)/\phi_n$. Thus, 
the problem of minimizing $f$ over $C$ is a special case of
Problem~\ref{prob:1}\ref{prob:1b} with $m=2$, and 
\eqref{e:newcoldwar} is a special case of Algorithm~\ref{algo:1}
with $e_{1,n}\equiv 0$ and $e_{2,n}\equiv 0$. 
Now let $(k_n)_{n\in\NN}$ be a strictly increasing sequence in 
$\NN$ and let $x\in\HH$ be such that 
$\overline{x}_{k_n}\weakly x$. Then, by
Corollary~\ref{c:1}\ref{c:1iii}, it remains to show that 
$x\in C\cap D$. We derive from Corollary~\ref{c:1}\ref{c:1ii} and 
\eqref{dFse7hY03a} that
$T_{2,{k_n}}\overline{x}_{k_n}-P_CT_{2,{k_n}}
\overline{x}_{k_n}\to 0$ and 
$\overline{x}_{k_n}-T_{2,{k_n}}\overline{x}_{k_n}\to 0$. Therefore
$C\ni P_CT_{2,{k_n}}\overline{x}_{k_n}=
(P_CT_{2,{k_n}}\overline{x}_{k_n}-T_{2,n}\overline{x}_{k_n})+
(T_{2,{k_n}}\overline{x}_{k_n}-\overline{x}_{k_n})+
\overline{x}_{k_n}\weakly x$ and, since $C$ is weakly closed, 
$x\in C$. On the other hand, 
$\|G\overline{x}_n-\overline{x}_{n}\|=
\|T_{2,n}\overline{x}_n-\overline{x}_{n}\|/\xi_n\leq
\|T_{2,n}\overline{x}_n-\overline{x}_{n}\|/\eta\to 0$. 
Since \ref{ex:borisiii}$\Rightarrow$\ref{ex:borisii}%
$\Leftrightarrow$\ref{ex:borisi} \cite[Proposition~16.20]{Livre1}
and \ref{ex:borisi} imply that $\Id-G$ is demiclosed at $0$
\cite[Proposition~29.41(vii)]{Livre1}, we conclude that
$x\in\Fix G=D$.
\end{proof}

\begin{remark}
Example~\ref{ex:boris} reverts to Polyak's classical result 
\cite[Theorem~1]{Poly69} in the case
when $(\forall n\in\NN)$ $\lambda_n=1$ 
and $(\forall j\in\{0,\ldots,n\})$ $\mu_{n,j}=\delta_{n,j}$. The
unrelaxed pattern $\lambda_n\equiv 1$ is indeed achievable
because $(\forall n\in\NN)$ $\lambda_n\in
[\varepsilon,(1-\varepsilon)(2-\xi_n/2)]$ and
$(1-\varepsilon)(2-\xi_n/2)\geq(1-\varepsilon)(2-(2-\eta)/2)
>(1-\eta/(2+\eta))(1+\eta/2)=1$.
\end{remark}


\begin{thebibliography}{99}  
\setlength{\itemsep}{1pt}
\small

\bibitem{Alva00}
F. Alvarez, 
On the minimizing property of a second order dissipative system in
Hilbert spaces,
{\em SIAM J. Control Optim.,}
vol. 38, pp. 1102--1119, 2000.

\bibitem{Alva04}
F. Alvarez, 
Weak convergence of a relaxed and inertial hybrid
projection-proximal point algorithm for maximal monotone 
operators in Hilbert space,
{\em SIAM J. Optim.,}
vol. 14, pp. 773--782, 2004.

\bibitem{Alva01}
F. Alvarez and H. Attouch, 
An inertial proximal method for maximal monotone operators via
discretization of a nonlinear oscillator with damping, 
{\em Set-Valued Anal.,}
vol. 9, pp. 3--11, 2001.

\bibitem{Sico10}
H. Attouch, L. M. Brice\~no-Arias, and P. L. Combettes, 
A parallel splitting method for coupled monotone inclusions, 
{\em SIAM J. Control Optim.,} 
vol. 48, pp. 3246--3270, 2010.

\bibitem{Atto00}
H. Attouch, X. Goudou, and P. Redont,
The heavy ball with friction method, I. The continuous dynamical 
system: global exploration of the local minima of real-valued 
function by asymptotic analysis of a dissipative dynamical
system,
{\em Commun. Contemp. Math.,}
vol. 2, pp. 1--34, 2000.

\bibitem{Atto16}
H. Attouch and J. Peypouquet, 
The rate of convergence of Nesterov's accelerated forward-backward
method is actually faster than $1/k^2$,
{\em SIAM J. Optim.,}
vol. 26, pp. 1824--1834, 2016.

\bibitem{Bail78} 
J. B. Baillon, R. E. Bruck, and S. Reich,
On the asymptotic behavior of nonexpansive mappings and semigroups
in Banach spaces,
{\em Houston J. Math.,}
vol. 4, pp. 1--9, 1978.

\bibitem{Barr16} 
D. Barro and E. Canestrelli,
Combining stochastic programming and optimal
control to decompose multistage stochastic optimization
problems,
{\em OR Spectrum,}
vol. 38, pp. 711--742, 2016.

\bibitem{Baus96} 
H. H. Bauschke and J. M. Borwein, On projection algorithms
for solving convex feasibility problems,
{\em SIAM Rev.,}
vol. 38, pp. 367--426, 1996.

\bibitem{Livre1} 
H. H. Bauschke and P. L. Combettes, 
{\em Convex Analysis and Monotone Operator Theory in Hilbert 
Spaces,} 2nd ed. 
Springer, New York, 2017.

\bibitem{Bore01} 
E. Borel,
{\em Le\c{c}ons sur les S\'eries Divergentes.}
Gauthier-Villars, Paris, 1901.

\bibitem{Borw91}
D. Borwein and J. Borwein, 
Fixed point iterations for real functions,
{\em J. Math. Anal. Appl.,}
vol. 157, pp. 112--126, 1991.

\bibitem{Borw17}
J. M. Borwein, G. Li, and M. K. Tam,
Convergence rate analysis for averaged fixed point iterations in
common fixed point problems,
{\em SIAM J. Optim.,}
vol. 27, pp. 1--33, 2017.

\bibitem{Botc15} 
R. I. Bo\c{t}, E. R. Csetnek, and C. Hendrich, 
Inertial Douglas-Rachford splitting for monotone inclusion problems,
{\em Appl. Math. Comput.,} 
vol. 256, pp. 472--487, 2015.

\bibitem{Brez98}
C. Brezinski and J.-P. Chehab, 
Nonlinear hybrid procedures and fixed point iterations,
{\em Numer. Funct. Anal. Optim.,}
vol. 19, pp. 465--487, 1998. 

\bibitem{Byrn14} 
C. L. Byrne,
{\em Iterative Optimization in Inverse Problems.}
CRC Press, Boca Raton, FL, 2014.

\bibitem{Cham15}
A. Chambolle and C. Dossal, 
On the convergence of the iterates of the ``Fast iterative
shrinkage/thresholding algorithm'', 
{\em J. Optim. Theory Appl.,}
vol. 166, pp. 968--982, 2015.

\bibitem{Else01} 
P. L. Combettes, 
Quasi-Fej\'erian analysis of some optimization algorithms, in
{\em Inherently Parallel Algorithms for Feasibility and Optimization}
(D. Butnariu, Y. Censor, and S. Reich, Eds.), pp. 115--152. 
Elsevier, New York, 2001.

\bibitem{Opti04}
P. L. Combettes, 
Solving monotone inclusions via compositions of nonexpansive
averaged operators,
{\em Optimization,} 
vol. 53, pp. 475--504, 2004.

\bibitem{Joca09}
P. L. Combettes, Iterative construction of the resolvent of a
sum of maximal monotone operators,
{\em J. Convex Anal.,} 
vol. 16, pp. 727--748, 2009.

\bibitem{Jmaa02}
P. L. Combettes and T. Pennanen, 
Generalized Mann iterates for constructing fixed points in
Hilbert spaces,
{\em J. Math. Anal. Appl.,}
vol. 275, pp. 521--536, 2002.

\bibitem{Siop15}
P. L. Combettes and J.-C. Pesquet, 
Stochastic quasi-Fej\'er block-coordinate fixed point iterations 
with random sweeping,
{\em SIAM J. Optim.,}
vol. 25, pp. 1221--1248, 2015.

\bibitem{Mapr17}
P. L. Combettes, S. Salzo, and S. Villa, 
Consistent learning by composite proximal thresholding,
{\em Math. Program.,}
published online 2017-03-25.

\bibitem{Smms05}
P. L. Combettes and V. R. Wajs, 
Signal recovery by proximal forward-backward splitting,
{\em Multiscale Model. Simul.,} 
vol. 4, pp. 1168--1200, 2005.

\bibitem{Jmaa15}
P. L. Combettes and I. Yamada, 
Compositions and convex combinations of averaged nonexpansive
operators,
{\em J. Math. Anal. Appl.,}
vol. 425, pp. 55--70, 2015.

\bibitem{Davi15}  
D. Davis,
Convergence rate analysis of primal-dual splitting schemes, 
{\em SIAM J. Optim.,}  
vol. 25, pp. 1912--1943, 2015.

\bibitem{Diaz67}
J. B. Diaz and F. T. Metcalf, 
On the structure of the set of subsequential limit points 
of successive approximations,
{\em Bull. Amer. Math. Soc.,}
vol. 73, pp. 516--519, 1967.

\bibitem{Dots70} 
W. G. Dotson, 
On the Mann iterative process,
{\em Trans. Amer. Math. Soc.,}
vol. 149, pp. 65--73, 1970.

\bibitem{Feje03} 
L. Fej\'er, 
Untersuchungen \"uber Fouriersche Reihen, 
{\em Math. Ann.,}
vol. 58, pp. 51--69, 1903.

\bibitem{Groe72} 
C. W. Groetsch, 
A note on segmenting Mann iterates,
{\em J. Math. Anal. Appl.,}
vol. 40, pp. 369--372, 1972.

\bibitem{Knop54} 
K. Knopp,
{\em Theory and Application of Infinite Series.} 
Blackie \& Son, London, 1954.

\bibitem{Kugl03} 
P. K\"ugler and A. Leit\~ao,
Mean value iterations for nonlinear elliptic Cauchy problems,
{\em Numer. Math.,}
vol. 96, pp. 269--293, 2003.

\bibitem{Lore14}
D. A. Lorenz and T. Pock, 
An inertial forward-backward algorithm for monotone inclusions, 
{\em J. Math. Imaging Vision,}
vol. 51, pp. 311--325, 2015.

\bibitem{Mann53} 
W. R. Mann, 
Mean value methods in iteration,
{\em Proc. Amer. Math. Soc.,}
vol. 4, pp. 506--510, 1953.

\bibitem{Mann79}
W. R. Mann, 
Averaging to improve convergence of iterative processes,
{\em Lecture Notes in Math.,}
vol. 701, pp. 169--179, 1979.

\bibitem{Mor62b} 
J. J. Moreau, 
Fonctions convexes duales et points proximaux dans un 
espace hilbertien,
{\em C. R. Acad. Sci. Paris S\'er. A Math.,}
vol. 255, pp. 2897--2899, 1962.

\bibitem{Outl69} 
C. Outlaw and C. W. Groetsch,
Averaging iteration in a Banach space,
{\em Bull. Amer. Math. Soc.,}
vol. 75, pp. 430--432, 1969.

\bibitem{Poly63} 
B. T. Polyak, 
Gradient methods for minimizing functionals,
{\em USSR Comput. Math. Math. Phys.,}
vol. 3, pp. 864--878, 1963.

\bibitem{Poly64} 
B. T. Polyak, 
Some methods of speeding up the convergence of iteration methods,
{\em USSR Comput. Math. Math. Phys.,}
vol. 4, pp. 1--17, 1964.

\bibitem{Poly69} 
B. T. Polyak, 
Minimization of unsmooth functionals,
{\em USSR Comput. Math. Math. Phys.,}
vol. 9, pp. 14--29, 1969.

\bibitem{Poly77} 
B. T. Polyak,
Comparison of the speed of convergence of single-step and
multistep algorithms of optimization when there are noises,
{\em Engrg. Cybernetics,}
vol. 15, pp. 6--10, 1977.

\bibitem{Rhoa74} 
B. E. Rhoades, 
Fixed point iterations using infinite matrices,  
{\em Trans. Amer. Math. Soc.,}
vol. 196, pp. 161--176, 1974.

\bibitem{Slav13} 
K. Slavakis and I. Yamada,
The adaptive projected subgradient method constrained by families
of quasi-nonexpansive mappings and its application to online
learning,
{\em SIAM J. Optim.,}
vol. 23, pp. 126--152, 2013.

\bibitem{Toep11} 
O. Toeplitz,
\"Uber allgemeine lineare Mittelbildungen,
{\em Prace Mat.-Fiz.,}
vol. 22, pp. 113--119, 1911.


\end{thebibliography}
\end{document}